\documentclass[graybox]{svmult}
\makeatletter
\newcommand\myparagraph{\@startsection{paragraph}{4}{\z@}%
	{3.25ex \@plus1ex \@minus.2ex}%
	{-1em}%
	{\normalfont\normalsize\bfseries}}
\makeatother
\DeclareMathAlphabet{\mymathcal}{OMS}{cmsy}{m}{n}

\usepackage{mathptmx}       
\usepackage{helvet}         
\usepackage{courier}        
\usepackage{type1cm}        
\usepackage{xspace}
\usepackage{makeidx}         
\usepackage{graphicx}        
\usepackage{multicol}        
\usepackage[bottom]{footmisc}
\usepackage{amsmath}
\usepackage{amssymb}
\usepackage{caption}
\newcommand{\bigtilde}[1]{\ensuremath{\widetilde{\mbox{#1}}}‌​}
\newcommand{\SVRF}{\bigtilde{SVRF}\xspace}
\newcommand{\SSVRF}{\bigtilde{SSVRF}\xspace}
\newlength{\intextsepsave}

\usepackage[usenames,dvipsnames]{xcolor}
\definecolor{purple}{rgb}{0.3,0.0,.4}

\newcommand{\beas}{\begin{eqnarray*}}
\newcommand{\eeas}{\end{eqnarray*}}
\newcommand{\bea}{\begin{eqnarray}}
\newcommand{\eea}{\end{eqnarray}}
\newcommand{\beq}{\begin{equation}}
\newcommand{\eeq}{\end{equation}}
\newcommand{\bit}{\begin{itemize}}
\newcommand{\eit}{\end{itemize}}
\newcommand{\ben}{\begin{enumerate}}
\newcommand{\een}{\end{enumerate}}

\newcommand{\ba}{\begin{array}}
\newcommand{\ea}{\end{array}}
\newcommand{\bbm}{\begin{bmatrix}}
\newcommand{\ebm}{\end{bmatrix}}

\newcommand{\eg}{{\it e.g.}}
\newcommand{\ie}{{\it i.e.}}

\newcommand{\reals}{{\mbox{\bf R}}}

\newcommand{\sym}{{\mbox{\bf S}}}  

\newcommand{\rank}{\mbox{\textrm{Rank}}}

\newcommand{\tr}{\mathop{\bf tr}}

\newcommand{\Expect}{\mathop{\bf E{}}}
\newcommand{\Prob}{\mathop{\bf Prob}}

\newcommand{\sign}{\mathop{\bf sign}}

\newcommand{\dist}{\mathop{\bf dist{}}}

\newcommand{\var}{\mathop{\bf var}}

\newcommand{\conv}{\mathop{\bf conv}}

\newcommand{\argmax}{\mathop{\rm argmax}}

\newcommand{\indicator}{{\bf 1}}

\makeatletter
\newlength \figwidth
\if@twocolumn
\else
\fi
\makeatother

\makeindex             
\usepackage{graphicx,amsmath,comment,framed,algorithm,algorithmic,url}
\usepackage[export]{adjustbox}
\usepackage{caption}
\newcommand{\amscode}[1]{\par\addvspace\baselineskip
	\noindent{\bf AMS subject classifications.}\enspace\ignorespaces#1}
\begin{document}

\title*{Frank-Wolfe Style Algorithms for Large Scale Optimization}
\author{Lijun Ding and Madeleine Udell}
\institute{Lijun Ding \at Operations Research and Information Engineering, Cornell University \email{ld446@cornell.edu}
\and Madeleine Udell \at Operations Research and Information Engineering, Cornell University \email{udell@cornell.edu}}
%
%
\maketitle

\abstract*{We introduce a few variants on
	Frank-Wolfe style algorithms suitable for large scale optimization.
	We show how to modify the standard Frank-Wolfe algorithm
	using stochastic gradients, approximate subproblem solutions, and sketched
	decision variables in order to scale to enormous problems while
	preserving (up to constants) the optimal convergence rate $\mymathcal{O}(\frac{1}{k})$.}

\abstract{We introduce a few variants on
	Frank-Wolfe style algorithms suitable for large scale optimization.
	We show how to modify the standard Frank-Wolfe algorithm
	using stochastic gradients, approximate subproblem solutions, and sketched
	decision variables in order to scale to enormous problems while
	preserving (up to constants) the optimal convergence rate $\mymathcal{O}(\frac{1}{k})$.}
\amscode{90C06, 90C25.}
\section{Introduction}
\label{sec:1}
This chapter describes variants on
Frank-Wolfe style algorithms suitable for large scale optimization.
Frank-Wolfe style algorithms enforce constraints by solving a linear optimization
problem over the constraint set at each iteration,
while competing approaches, such as projected or proximal gradient algorithms,
generally require projection onto the constraint set.
For important classes of constraints, such as the unit norm ball
of the $\ell_1$ or nuclear norm, linear optimization over the constraint set
is much faster than projection onto the set.
This paper provides a gentle introduction to three ideas that can be used to
further improve the performance of Frank-Wolfe style algorithms for large scale optimization:
stochastic gradients, approximate subproblem solutions, and sketched
decision variables. Using these ideas, we show how to modify the standard Frank-Wolfe algorithm
in order to scale to enormous problems while
preserving (up to constants) the optimal convergence rate.

To understand the challenges of huge scale optimization,
let us start by recalling the original Frank-Wolfe algorithm.
The Frank-Wolfe algorithm is designed to solve problems of the form
\beq \label{pr1}
\ba{ll}
\mbox{minimize} & f(x)\\
\mbox{subject to} &  x\in \Omega,\\
\ea
\eeq
where $f$ is a real valued convex differentiable function from $\reals^n$ to $\reals$, and the set $\Omega$ is a nonempty compact convex set in $\reals^n$.
Throughout the sequel, we let $x^\star \in \arg \min_{x\in \Omega} f(x)$ be an arbitrary solution to \eqref{pr1}.

The Frank-Wolfe algorithm is presented as Algorithm~\ref{fw} below.
At each iteration, it computes the gradient of the objective ${\nabla}f(x)$ at the current iterate $x$,
and finds a feasible point $v \in \Omega$ which maximizes ${\nabla}f(x)^Tv$.
The new iterate is taken to be a convex combination of the previous iterate and the point $v$.
\begin{algorithm}
	\caption{Frank-Wolfe Algorithm \label{fw}}
	\begin{algorithmic}[1]
		\STATE {\bf Input:} Objective function $f$ and feasible region $\Omega$
		\STATE{\bf Input:} A feasible starting point $x_{-1}\in \Omega$
		\STATE{\bf Input:} Stepsize sequence $\gamma_k$ and tolerance level $\varepsilon>0$
		\FOR{$k=0,1,2, \ldots$}
		\STATE Compute $ v_k = \arg \min_{v\in \Omega} {\nabla}f(x_{k-1})^Tv.$ \label{l5}
		\IF{$(x_{k-1}-v_k)^T\nabla f(x_{k-1}) \leq \varepsilon $} \label{lstop}
		\STATE break
		\ENDIF
		\STATE Update $x_k = (1-\gamma_k)x_{k-1}+\gamma_k v_k$.

		\ENDFOR
		\STATE {\bf Output:} The last iteration result $x$
	\end{algorithmic}
\end{algorithm}

The Frank-Wolfe algorithm can be used for optimization with matrix variables as well.
With some abuse of notation, when $x,\nabla f(x)$, and $v$ are matrices rather than vectors,
we use the inner product $\nabla f(x)^Tv$ to denote the matrix trace inner product $\tr(\nabla f(x)^Tv)$.

\myparagraph{Linear Optimization Subproblem.} The main bottleneck in implementing Frank-Wolfe
is solving the linear optimization subproblem in Line \ref{l5} above:
\beq
\ba{ll} \label{sp1}
\mbox{minimize} & \nabla f(x_{k-1})^Tv\\
\mbox{subject to} &  v\in \Omega.\\
\ea
\eeq
Note that the objective of the subproblem \eqref{sp1} is linear even though the constraint set $\Omega$ may not be.  Since $\Omega$ is compact, the solution to subproblem \eqref{sp1} always exists.  Subproblem \eqref{sp1} can easily be solved when the feasible region has atomic structure \cite{atm}. We give three examples here.
\begin{itemize}
	\item \emph{The feasible region is a one norm ball.} For some $\alpha>0$,  $$\Omega = \{x\in \reals^n \mid\|x\|_1\leq \alpha\}.$$ Let $\{e_i\}_{i=1}^n$ to be the standard basis in $\reals ^n$,
	$S = \argmax_{i} |\nabla f(x_{k-1})^Te_i|$ and $s_i = \sign (\nabla f(x_{k-1})^Te_i)$.
	The solution $v$ to subproblem \eqref{sp1} is any vector in the convex hull of
	$\{-\alpha s_ie_i\mid i\in S \}$:
	$$ v \in \conv(\{-\alpha s_ie_i\mid i\in S \}).$$
	In practice, we generally choose $v = -\alpha s_ie_i$ for some $i\in S$.
	\item \emph{The feasible region is a nuclear norm ball.} For some $\alpha>0$,
	$$\Omega = \{X\in \reals^{m\times n}\mid \|X\|_* \leq \alpha \},$$
	where $\| \cdot\|_*$ is the nuclear norm, \ie, the sum of the singular values.
	Here $v$, $x_{k-1}$, and $\nabla f(x_{k-1})$ are matrices in $\reals ^{m\times n}$,
	and we recall that the objective in Problem \eqref{sp1}, $v^T\nabla f(x_{k-1})$,
	should be understood as the matrix trace inner product $\tr(\nabla f(x_{k-1})^Tv)$.
	Subproblem \eqref{sp1} in this case is
	\beq
	\ba{ll}
	\mbox{minimize} & \tr(\nabla f(x_{k-1})^Tv)\\
	\mbox{subject to} &  \|v\|_*\leq \alpha.\\
	\ea
	\eeq
	Denote the singular values of $\nabla f(x_{k-1})$ as
	$\sigma_1\geq \dots, \geq \sigma_{\min(m,n)}$ and the corresponding singular vectors
	as $(u_1,v_1), \dots, (u_{\min(m,n)},v_{\min(m,n)})$.
	Let $S = \{i \mid \sigma_i = \sigma_1\}$ be the set of indices with maximal singular value.
	Then the solution to problem \eqref{sp1} is the convex hull of the singular vectors with
	maximal singular value, appropriately scaled:
	\[
	\conv(\{ -\alpha u_iv_i^T\mid i \in S \}).
	\]
	In practice, we often take the solution $-\alpha u_1^Tv_1$.
	This solution is easy to compute compared to the full singular value decomposition.
	Specifically, suppose $\nabla f(x_{k-1})$ is sparse,
	and let $s$ be the number of non-zero entries in $\nabla f(x_{k-1})$.
	For any tolerance level $\epsilon>0$,
	the number of arithmetic operations required to compute the top singular tuple $(u_1,v_1)$
	using the Lanczos algorithm such that $u_1^T\nabla f(x_{k-1}) v_1 \geq \sigma_1 -\epsilon$
	is at most $\mymathcal{O}(s\frac{\log(m+n)\sqrt{\sigma_1}}{\sqrt{\epsilon}})$ with high probability \cite{KW}.

	\item \emph{The feasible region is a restriction of a nuclear norm ball.}
	For some $\alpha>0$,
	\[
	\Omega = \{X \in \reals^{n\times n}\mid \|X\|_* \leq \alpha,~ X\succeq 0\},
	\]
	where $X\succeq 0$ means $X$ is symmetric and positive semidefinite, \ie, every eigenvalue of $X$ is nonnegative.
	In this case the objective in problem \eqref{sp1} $v^T\nabla f(x_{k-1})$ should
	be understood as $\tr(\nabla f(x_{k-1})^Tv)$, where $v$, $x_{k-1}$, and $\nabla f(x_{k-1})$
	are matrices in $\sym ^{n}$.
	The subproblem \eqref{sp1} in this case is just
	\beq
	\ba{ll}
	\mbox{minimize} & \tr(\nabla f(x_{k-1})^Tv)\\
	\mbox{subject to} &  \|v\|_*\leq \alpha \\
	                  & v\succeq 0.\\
	\ea
	\eeq
	Denote the eigenvalues of $\nabla f(x_{k-1})$ as $\lambda_1\geq \dots, \geq \lambda_{n}$
	and the corresponding eigenvectors as $v_1, \dots, v_{n}$.
	Let $S = \{i \mid \lambda_i = \lambda_n\}$ be the set of indices with smallest eigenvalue.
	Then the solution to Problem \eqref{sp1} is simply $0$ if $\lambda_{n}\geq 0$,
	while if $\lambda_{n}\leq 0$, the solution set consists of
	the convex hull of the eigenvectors with smallest eigenvalue, appropriately scaled:
	\[
	\conv(\{ \alpha v_iv_i^T\mid i \in S \}).
	\]

	In practice, we generally take $\alpha v_nv_n^T$ as a solution (if $\lambda_n\leq 0$).
	As in the previous case, this solution is easy to compute compared to the full eigenvalue decomposition.
	Specifically, suppose $\nabla f(x_{k-1})$ is sparse,
	and let $s$ be the number of non-zero entries in $\nabla f(x_{k-1})$.
	For any tolerance level $\epsilon>0$,
	the number of arithmetic operations required to compute the eigenvector $v_n$
	using the Lanczos algorithm such that $v_1^T \nabla f(x_{k-1}) v_1 \leq \lambda_{n} +\epsilon$
	is at most $\mymathcal{O}(s\frac{\log(2n)\sqrt{\max(|\lambda_1|,|\lambda_n|)}}{\sqrt{\epsilon}})$ with high probability \cite[Lemma 2]{hazan2}.
\end{itemize}

Thus, after $k$ iterations of the Frank-Wolfe algorithm,
the sparsity or rank of the iterate $x_k$ in the above three examples is bounded by $k$.
This property has been noted and exploited by many authors \cite{rob, hazan2, Jaggi}.

The stopping criterion in Line \ref{lstop}  of Algorithm \ref{fw} bounds the suboptimality $f(x_{k-1}) - f(x^\star)$, where $x^\star \in \arg \min _{x\in \Omega } f(x)$. Indeed,
\begin{align*}
f(x_{k-1}) - f(x^\star) & \leq (x_{k-1} - x^\star) ^T\nabla f(x_{k-1}) \\
& \leq (x_{k-1} - v_k)^T\nabla f(x_{k-1}),
\end{align*}
where the first inequality is due to convexity and the second line is due to optimality of $v_k$.

\myparagraph{Matrix Completion.}
To illustrate our previous points, let's consider the example of matrix completion.
Keep this example in mind: we will return to this problem again in the coming sections
to illustrate our methods.

We consider the optimization problem
\beq
\ba{ll}\label{op3}
\mbox{minimize} & f(\mymathcal{A}X)\\
\mbox{subject to} &  \|X\|_*\leq \alpha,\\
& X\in \mymathcal{S}\\
\ea
\eeq
with variable $X \in \reals^{m\times n}$.
Here $\mymathcal{A}:\reals^{m\times n} \rightarrow \reals^d$ is a linear map and $\alpha>0$ is a positive constant. The set $\mymathcal{S}$ represents some additional information of the underlying problem. In this book chapter, the set $\mymathcal{S}$ will be either $\reals^{m \times n}$ or $\{X \in \reals^{n\times n}\mid  X\succeq 0\}$. In the first case, the feasible region of  Problem \eqref{op3} is just the nuclear norm ball. In the second case,  the feasible region is a restriction of the nuclear norm. In either case, the linear optimization subproblem can be solved efficiently as we just mentioned. The function $f:\reals^d \rightarrow \reals$ is a loss function that penalizes the misfit
between the predictions $\mymathcal{A}X$ of our model and our observations from the matrix.

For example, suppose we observe matrix entries $c_{ij}$ with indices in
$\mymathcal{O} \subset \{1,\dots, m\} \times \{1,\dots,n\}$ from a matrix $X^0\in \reals^{m\times n}$ corrupted by Gaussian noise $E$:
$$
c_{ij} =(X^{0})_{ij} + E_{ij}, \qquad E_{ij} \stackrel{iid}{\sim} N(0,\sigma^2)
$$
for some $\sigma>0$.
A maximum likelihood formulation of problem \eqref{op3} to recover $X^0$ would be
\beq
\ba{ll}\label{op4}
\mbox{minimize} & \sum_{(i,j)\in \mymathcal{O}}(x_{ij}-c_{ij})^2\\
\mbox{subject to} &  \|X\|_*\leq \alpha.\\
\ea
\eeq
To rewrite this problem in the form of \eqref{op3},
we choose $\mymathcal{A}$ so that $(\mymathcal{A}X)_{ij} = x_{ij}$ for $(i,j) \in \mymathcal{O}$, so
the number of observations $d$ is the cardinality of $\mymathcal{O}$. Since there is no additional information of $X^0$, we set $\mymathcal{S}=\reals^{n\times m}$.
The objective $f$ is  a sum of quadratic losses in this case.

Since the constraint region $\Omega$ is a nuclear norm ball when $\mymathcal{S} = \reals^{m\times n}$, we can apply Frank-Wolfe to this optimization problem.
The resulting algorithm is shown as Algorithm~\ref{fw1}.
Here Line \ref{l6} computes the singular vectors with largest singular value, and
Line \ref{l7} exploits the fact that that at each iteration we can choose a rank one update.
\begin{algorithm}
	\caption{Frank-Wolfe Algorithm Applied to Matrix Completion with nuclear ball constraint only\label{fw1}}
	\begin{algorithmic}[1]
		\STATE {\bf Input:} Objective function $f$ and $\alpha >0$
		\STATE{\bf Input:} A feasible starting point $\|X_{-1}\| \leq \alpha$
		\STATE{\bf Input:} Stepsize sequence $\gamma_k$
		\FOR{$k=0,1,2, \ldots,K$}
		\STATE Compute $ (u_k,v_k) $, the top singular vectors of ${\nabla}f(\mymathcal {A}X_k)$. \label{l6}
		\IF{$\tr((X_{k-1}+\alpha u_kv_k^T)^T\nabla f(X_{k-1}))\leq \varepsilon $}
		\STATE break the for loop.

		\ENDIF
		\STATE Update $X_k = (1-\gamma_k)X_{k-1}-\gamma_k \alpha u_kv_k^T$.  \label{l7}

		\ENDFOR
		\STATE {\bf Output:} The last iteration result $X_K$
	\end{algorithmic}
\end{algorithm}

However, there are three main challenges in the large scale setting
that can pose difficulties in applying the Frank-Wolfe algorithm:
\begin{enumerate}
	\item Solving the linear optimization subproblem \eqref{sp1} exactly,
	\item computing the gradient $\nabla f$, and
	\item storing the decision variable $x$.
\end{enumerate}

To understand why each of these steps might present a difficulty, consider again the matrix completion case with nuclear ball constraint only.
\begin{enumerate}
	\item Due to Galois theory, it is not possible to exactly compute the top singular vector, even in exact arithmetic.
	Instead, we rely on iterative methods such as the QR algorithm with shifts, or the Lanczos method,
	which terminate with some approximation error.
	What error can we allow in an approximate solution of the linear optimization subproblem \eqref{sp1}?
	How will this error affect the performance of the Frank-Wolfe algorithm?
	\item In many machine learning and statistics problems,
	the objective $f(X) = \sum_{i=1}^d f_i(X)$ is a sum over $d$ observations,
	and each $f_i$ measures the error in observation $i$.
	As we collect more data, computing $\nabla f$ exactly becomes more difficult,
	but approximating $\nabla f$ is generally easy.
	Can we use an approximate version of $\nabla f$ instead of the exact gradient?
	\item Storing $X$, which requires $m \times n$ space in general, can be costly if $n$ and $m$ are large.
	One way to avoid using $\mymathcal{O}(mn)$ memory is to store each updates $(u_k,v_k)$.
	But this approach still uses $\mymathcal O(mn)$ memory when the number of iterations $K \geq \min(m,n)$.
	Can we exploit structure in the solution $X^\star$ to reduce the memory requirements?
\end{enumerate}
We provide a crucial missing piece to address the first challenge, and
a gentle introduction to the ideas needed to tackle the second and third challenges.
Specifically, we will show the following.
\begin{enumerate}
	\item Frank-Wolfe type algorithms still converge when we use an approximate oracle
	to solve the linear optimization subproblem \eqref{sp1}.
	In fact, the convergence rate is preserved up to a multiplicative user-specified constant.
	\item Frank-Wolfe type algorithms still converge when the gradient is replaced by an approximate gradient,
	and the convergence rate is preserved in expectation.
	\item Frank-Wolfe type algorithms are amenable to a matrix sketching procedure which
	can be used to reduce memory requirements, and the convergence rate is not affected.
\end{enumerate}

Based on these ideas, we propose two new Frank-Wolfe Style algorithms which we call
SVRF with approximate oracle (\SVRF, pronounced as ``tilde SVRF''),
and Sketched \SVRF (\SSVRF).
They can easily scale to extremely large problems.

The rest of this chapter describes how \SSVRF
addresses the three challenges listed above.
To address the first challenge, we augment the Frank-Wolfe algorithm
with an approximate oracle for the linear optimization subproblem, and prove a
convergence rate in this setting. Numerical experiments confirm that using
an approximate oracle reduces the time necessary to achieve a given error tolerance.
To address the second challenge,
we then present a Stochastic Variance Reduced Frank-Wolfe (SVRF) algorithm with approximate oracle, \SVRF.
Finally, we show how to use
the matrix sketching procedure of \cite{sketch}
to reduce the memory requirements of the algorithm.
We call the resulting algorithm \SSVRF.

\myparagraph{Notation.}
We use $\|\cdot\|$ to denote the Euclidean norm when the norm is applied to a vector,
and to denote the operator norm (maximum singular value) when applied to a matrix.
We use $\|\cdot\|_F$ to denote the Frobenius norm and $\|\cdot\|_*$ to denote the nuclear norm
(sum of singular values).
The transpose of a matrix $A$ and a vector $v$ is denoted as $A^T$ and $v^T$. The trace of a  matrix $A\in \reals^{n\times n}$ is the sum of all its diagonals, \ie,
$\tr(A) =\sum_{i=1}^n A_{ii}$.
The set of symmetric matrices in $\reals^{n\times n}$ is denoted as $\sym^n$.
We use $X\succeq 0$ to mean that $X$ is symmetric and positive semidefinite (psd).
A convex function $f:\reals^n \rightarrow \reals$ is $L$-smooth if $\|\nabla f(x)- \nabla f(y)\| \leq L\| x-y\|$
for some finite $L \geq 0$.
The diameter $D$ of a set $\Omega \subset \reals ^n$ is defined as $D = \sup_{x,y\in \Omega} \|x-y\|$.
For an arbitrary matrix $Z$, we define $[Z]_r$ to be the best rank $r$ approximation of $Z$ in Frobenius norm. For a linear operator $\mymathcal{A}:\reals^{m\times n} \rightarrow \reals^{l}$, where $\reals^{m\times n}$ and $\reals^{n}$ are equipped with the trace inner product and the Euclidean inner product, the adjoint of $\mymathcal{A}$ is denotes as $\mymathcal{A}^*:\reals^{l} \rightarrow \reals^{m\times n}$.

\section{Frank-Wolfe with Approximate Oracle}\label{s2}
In this section, we address the first challenge:
the linear optimization subproblem \eqref{sp1} can only be solved approximately.
Most of the ideas in this section are drawn from \cite{Jaggi};
we include this introduction for the sake of completeness.

We will show that the Frank-Wolfe algorithm with approximate subproblem oracle
converges at the same rate as the one with exact subproblem oracle
up to a user-specified multiplicative constant.
\subsection{Algorithm and convergence}
As before, we seek to solve Problem \eqref{pr1},
\begin{equation*}
\ba{ll}
\mbox{minimize} & f(x)\\
\mbox{subject to} &  x\in \Omega.\\
\ea
\end{equation*}

Let us introduce Algorithm \ref{fwap}, which we call Frank-Wolfe with approximate oracle.
The only difference from the original Frank-Wolfe algorithm is the tolerance $\epsilon_k>0$:
in Line \ref{l5ap}, we compute an approximate solution with tolerance $\epsilon_{k}$
rather than an exact solution.
\begin{algorithm}
	\caption{Frank-Wolfe with approximate oracle \label{fwap}}
	\begin{algorithmic}[1]
		\STATE {\bf Input:} Objective function $f$ and feasible region $\Omega$
		\STATE{\bf Input:} A feasible starting point $x_{-1}\in \Omega$
		\STATE{\bf Input:} Stepsize sequence $\gamma_k$ and tolerance level $\varepsilon$ and error sequence $\epsilon_k>0$
		\FOR{$k=0,1,2, \ldots$}
		\STATE Compute $ v_k$ such that $\nabla f(x_{k-1})^T v_k\leq  \min_{v\in \Omega} {\nabla}f(x_{k-1})^Tv+\epsilon_{k}.$ \label{l5ap}
		\IF{$(x_{k-1}-v_k)^T\nabla f(x_{k-1}) \leq \varepsilon $}
		\STATE break
		\ENDIF
		\STATE Update $x_k = (1-\gamma_k)x_{k-1}+\gamma_k v_k$. \label{upap}

		\ENDFOR
		\STATE {\bf Output:} The last iteration result $x_k$
	\end{algorithmic}
\end{algorithm}

There are a few variants on this algorithm that use different line search methods.
The next iterate might be the point on the line determined by $x_{k-1}$ and $v_k$ with lowest objective value,
or the point with best objective value on the polytope with vertices $x_{-1},v_1,\dots,v_k$.
These variants may reduce the total number of iterations at the cost of an
increased per-iteration complexity.
When memory is plentiful and line search or polytope search is easy to implement,
these techniques can be employed; otherwise, a predetermined stepsize rule, \ie, $\gamma_k$ is determined as an input, \eg,$\gamma_k = \frac{2}{k+2}$ or $\gamma_k$ is a constant, might be preferred.
All these techniques enjoy the same complexity bounds as Algorithm \ref{fwap} since within an iteration,
starting from the same iterate $x_k$,
the objective is guaranteed to decrease at least as much
under each of these line search rules
as using the predetermined stepsize rule in Algorithm \ref{fwap}.

The following theorem gives a guarantee on the primal convergence of the objective value when $f$ is $L$-smooth.
\begin{theorem}\label{t0}
	Given an arbitrary $\delta>0$, if $f$ is $L$-smooth, $\Omega$ has diameter $D$, $\gamma_k = \frac{2}{k+2}$ and
	$\epsilon_{k}=\frac{LD^2}{2}\gamma_k\delta$,
	then the iterates $x_k$ of Algorithm \ref{fwap} 
	satisfy
	\begin{align}
	f(x_k) - f(x^\star) \leq \frac{2LD^2}{k+2}(1+\delta) \label{tt1}
	\end{align}
	where $x^\star \in \arg\min_{x\in \Omega } f(x)$.
\end{theorem}

To start, recall an equivalent definition of $L$-smoothness \cite[Theorem 2.1.5]{nesterov2013introductory}. For completeness,
we provide a short proof in the appendix.
\begin{proposition}\label{p0}
	If the real valued differentiable convex function $f$ with domain $\reals^n$ is
	$L$-smooth, 
	then for all $x,y \in \reals^n$,
	\[
	f(x)\leq f(y) + \nabla f(y)^T (x-y)+\frac{L}{2}\|x-y\|^2.
	\]
\end{proposition}
\begin{proof}[Proof of Theorem \ref{t0}]

	Let $v^\star_k\in \arg\min_{v\in \Omega} f(x_{k-1})^Tv$ in Line \ref{l5ap}.
	Using the update equation $x_{k} = x_{k-1} +  \gamma_k(v_k- x_{k-1})$, we have
	\beq \label{prt0}
	\ba{ll}
	f(x_k) -f(x^\star) & \leq f(x_{k-1}) -f(x^\star)+ \nabla f(x_{k-1})^T(v_k-x_{k-1}) \gamma_k + \frac{L}{2} \gamma_k^2\|v_k - x_{k-1}\|^2\\
	 &\leq f(x_{k-1}) -f(x^\star)+ \nabla f(x_{k-1})^T(v_k-x_{k-1})\gamma_k + \frac{LD^2}{2}\gamma_k^2\\
	 &\leq f(x_{k-1}) -f(x^\star) + \nabla f(x_{k-1})^T(v^\star_k - x_{k-1}) \gamma_k + \frac{LD^2}{2} \gamma_k^2(1+\delta)\\
	& \leq  f(x_{k-1}) -f(x^\star) +\nabla f(x_{k-1})^T(x^\star -x_{k-1}) \gamma_k + \frac{LD^2}{2} \gamma_k^2(1+\delta)\\
	 &\leq  (1-\gamma_k)( f(x_{k-1})-f(x^\star)) +\frac{LD^2}{2}\gamma_k^2(1+\delta).\\
	\ea
	\eeq
The first inequality is due to Proposition \ref{p0}. The second inequality uses the diameter $D$ of $\Omega$ and the fact that $x_k$ is feasible since $\gamma_k\in (0,1)$ and $x_k$ is a convex combination of points in the convex set $\Omega$. The third inequality uses the bound on the suboptimality of $v_k$ in Line \ref{l5ap}. The fourth inequality uses the optimality of $v_k^\star$  for $\min_{v\in\Omega} v^T\nabla f(x_{k-1})$ and fifth uses convexity of $f$.  The conclusion of the above chain of inequalities is
	\beq
	\ba{ll}\label{e1}
	f(x_k)-f(x^\star)\leq (1-\gamma_k)( f(x_{k-1})-f(x^\star)) +\frac{LD^2}{2}\gamma_k^2(1+\delta).
	\ea
	\eeq
	
	Now we prove inequality \eqref{tt1} by induction.
	The base case $k=0$ follows from \eqref{prt0} since $\gamma_0= 1$.
	Now suppose inequality \eqref{tt1} is true for $k\leq s$. Then for $k=s+1$,
	\beq
	\ba{ll}
	f(x_{s+1}) -f(x^\star)&
	\leq (1-\frac{2}{s+2+1})( f(x_{s})-f(x^\star)) +\frac{LD^2}{2}\big(\frac{2}{s+2+1}\big )^2(1+\delta)\\
	&=  \frac{s+1}{s+2+1} ( f(x_{s})-f(x^\star)) +\frac{LD^2}{2}\big(\frac{2}{s+2+1}\big)^2(1+\delta)\\
	&\leq   \big( \frac{s+1}{s+2+1}\frac{2}{s+2} +\frac{2}{(s+2+1)^2}\big) LD^2(1+\delta )\\
&	=  \big (\frac{2s+2}{s+2} +\frac{2}{s+2+1} \big) \frac{LD^2}{s+2+ 1}(1+\delta) \\
&	\leq  \big (\frac{2s+2+2}{s+2} \big)  \frac{LD^2}{s+2+1}(1+\delta)\\
&	=\frac{2LD^2}{s+1 + 2} (1+\delta ).
	\ea
	\eeq
	We use \eqref{e1} 
	in the first inequality and the induction hypothesis in the second inequality to bound the term $f(x_{s})-f(x^\star)$.
	The last line completes the induction.
\end{proof}
\subsection{Numerics}

In this subsection, we demonstrate that Frank-Wolfe is robust to using an approximate oracle through numerical experiments.

The specific problem we will use as our case study is the following symmetric matrix completion problem which is a special case of Problem \eqref{op3}.
The symmetric matrix completion problem seeks to recover an underlying matrix $X^0 \succeq 0$ from a few noisy entries of $X^0$. Specifically, let $C =X^0 +E$ be a matrix of noisy observations of $X^0$, where $E$ is a symmetric noise matrix. For each $i\geq j$, we observe $C_{ij}$ independently with probability $p$. The quantity $p$ is called the sample rate.

Let $\mymathcal{O}$ be the set of observed entries and $m$ be the number of entries observed.
Note that if $(i,j)\in \mymathcal{O}$, $(j,i)\in \mymathcal{O}$ as well since our matrices are all symmetric.

The optimization problem we solve to recover $X^0$ is
\beq
\ba{ll}\label{op6}
\mbox{minimize} &f(X):\,=\frac{1}{2} \|P_{\mymathcal{O}}(X)-P_{\mymathcal{O}}(C)\|_F^2\\
\mbox{subject to} &  \|X\|_*\leq \alpha, \\& X\succeq 0.\\
\ea
\eeq
Here the projection operator $P_\mymathcal{O}: \sym^n\rightarrow \reals^m$  is $$[P_{\mymathcal{O} }(Y)]_{ij} = \begin{cases}
Y_{ij} , &\text{if }  (i,j)\in \mymathcal{O} \\
0, & \text{if } (i,j)\notin \mymathcal{O}.
\end{cases}$$ for any $Y\in \sym^n$. By letting $\mymathcal{A}= P_{\mymathcal{O}}$, the set $\mymathcal{S}=\{X \in \reals^{n\times n}\mid X\succeq 0\}$ and $f(\cdot) = \|\cdot \|_F^2$, we see it is indeed a special case of Problem \eqref{op3}.

The gradient at $X_{k}$ is $\nabla f(X_k) = P_{\mymathcal{O}}(X_{k})-P_{\mymathcal{O}}(C) $. As we discussed in the introduction,  a solution to the linear optimization subproblem is
$$
V_k = \begin{cases}
\alpha v_nv_n^T,&\text{if }\lambda_{n}(\nabla f(X_{k-1})\leq 0\\
0,&\text{if }\lambda_{n}( \nabla f(X_{k-1}))>0
\end{cases}
$$
where $\lambda_{n}(\nabla f(X_{k-1}))$ is the smallest eigenvalue of $\nabla f(X_{k-1})$.

When the sample rate $p<1$ is fixed, \ie, independent of dimension $n$, the probability we observe all entries on the diagonal of $C$ is very small.
Hence the matrix $\nabla f(X_{k-1})$ is very unlikely to be positive definite, for any $k$.
(Recall that a positive definite matrix has positive diagonal.)
Let us suppose that at least one entry on the diagonal is not observed, so that $\lambda_{n}(\nabla f(X_{k-1}))\leq 0$ for every $k$.
Thus Line \ref{l5ap} of Algorithm \eqref{fwap} reduces to finding an approximate eigenvector $v$ such that

\beq
\ba{ll}\label{ap1}
\alpha v^T\nabla f(X_{k-1})v \leq  \alpha  \lambda_{n}(\nabla f(X_{k-1}))+ \epsilon_{k}.
\ea
\eeq

However, the solver ARPACK \cite{lehoucq1998arpack}, which is the default solver for iterative eigenvalue problems in a variety of languages (\eg, \texttt{eigs} in Matlab),
does not support specifying the approximation error in the form of \eqref{ap1}.
Instead, for a given tolerance $\xi_k$, it finds an approximate vector $v\in \reals^n$ with unit two norm, \ie, $\|v\|=1$, and an approximate eigenvalue $\lambda\in \reals $, such that
$$ \|\nabla f(X_{k-1})v- \lambda v\|\leq \xi_k\|\nabla f(X_{k-1})\|.$$

For simplicity, we assume that $\lambda$ returned by our eigenvalue solver is the true smallest eigenvalue $\lambda_n(\nabla f(X_{k-1}))$,
for any tolerance $\xi_k$. We will justify this assumption later through numerical experiments.
In this case, the error $\epsilon_k$ is upper bounded by
\beq
\ba{ll}\label{e2}
\xi_k \alpha  \|\nabla f(X_{k-1})\| \geq \epsilon_k.
\ea
\eeq

This upper bound turns out to be very conservative for large $\xi_k$:
$\xi_k \alpha\|\nabla f(X_{k-1})\| $ might be much larger than the actual error
$\epsilon_k = \alpha v^T\nabla f(X_{k-1})v -\alpha  \lambda_{n}(\nabla f(X_{k-1}))$,
as we will see later.

In the experiments, we set the dimension $n=1000$ and generated $X^0 = WW^T$,
where $W\in \reals^{n\times r}$ had independent standard normal distributed entries.
We then added symmetric noise $E = \frac{1}{10} \times  (L+L^T)$ to $X^0$ to get $C = X^0+E$,
where $L\in \reals^{n\times n}$ had independent standard normal entries.
We then sampled uniformly from the upper triangular part of $C$ (including the diagonal)
with probability $p=0.8$.

In each experiment we solved problem~\eqref{op6} with $\alpha = \|X^0\|_*$.
In real applications, one usually does not know $\|X^0\|_*$ in advance.
In that case, one might solve problem \eqref{op6} multiple times with different values of $\alpha$ and select the best $\alpha$ according to some criterion.

We ran $9$ experiments in total.
In each experiment, we chose a rank $r$ of $X^0$ in $\{10,50,100\}$ and
ran Frank-Wolfe with approximate oracle with constant tolerance $\xi_k \in \{10^{-15},10^{-5},1\}$
using the step size rule $\gamma_k =\frac{2}{k+2}$,
as required for Theorem \ref{t0},
and terminated each experiment after 30 seconds;
the qualitative performance of the algorithm is similar even after many more iterations.
We emphasize that within an experiment, the tolerance $\xi_k$ was the same for each iteration $k$.
See the discussion above Figure \ref{f3} for more details about the choice of $\xi_k$.

Figure \ref{f0} shows experimental results on the relationship between the relative objective $\frac{\|P_{\mymathcal{O}}(X_k)-P_{\mymathcal{O}}(C)\|_F^2}{\|P_{\mymathcal{O}}(C)\|^2_F}$ (on a log scale) and the actual clock time
under different combinations of rank $r$ and tolerance $\xi_k$.
For a fixed rank, the relative objective $\frac{\|P_{\mymathcal{O}}(X_k)-P_{\mymathcal{O}}(C)\|_F^2}{\|P_{\mymathcal{O}}(C)\|^2_F}$  evolves similarly for any tolerance.
When the underlying matrix has relatively high rank, using a lower tolerance allows faster convergence, at least for the moderate final relative objective achieved in these experiments.
The per iteration cost is summarized in Table \ref{table1}.
In fact, these plots show no advantage to using a tighter tolerance in any setting.

\captionsetup[figure]{labelfont={bf},labelformat={default},labelsep=period}
\captionsetup[table]{labelfont={bf},labelformat={default},labelsep=period}
\begin{table}[htbp]
	\vspace*{-.5cm}
	\centering
	\caption{Average per iteration time (seconds) of Algorithm \eqref{fwap} for problem \eqref{op6}.}
	\begin{tabular}{llll}
		\hline
		~&{$\rank(X^0)=10$}&{$\rank(X^0)=50$ }&{$\rank(X^0)=100$}
		\\
		\hline
		$\xi_k=10^{-15}$ & 0.1136& 0.1923& 0.2400\\
		$\xi_k=10^{-5}$&0.0997& 0.1376& 0.1840\\
		$\xi_k=1$ &  0.1017& 0.1099&0.1220\\
		\hline
	\end{tabular}\label{table1}
\end{table}

\begin{figure}
	\vspace*{-.5cm}
	\captionsetup{width= \textwidth}
	\includegraphics[width=\linewidth]{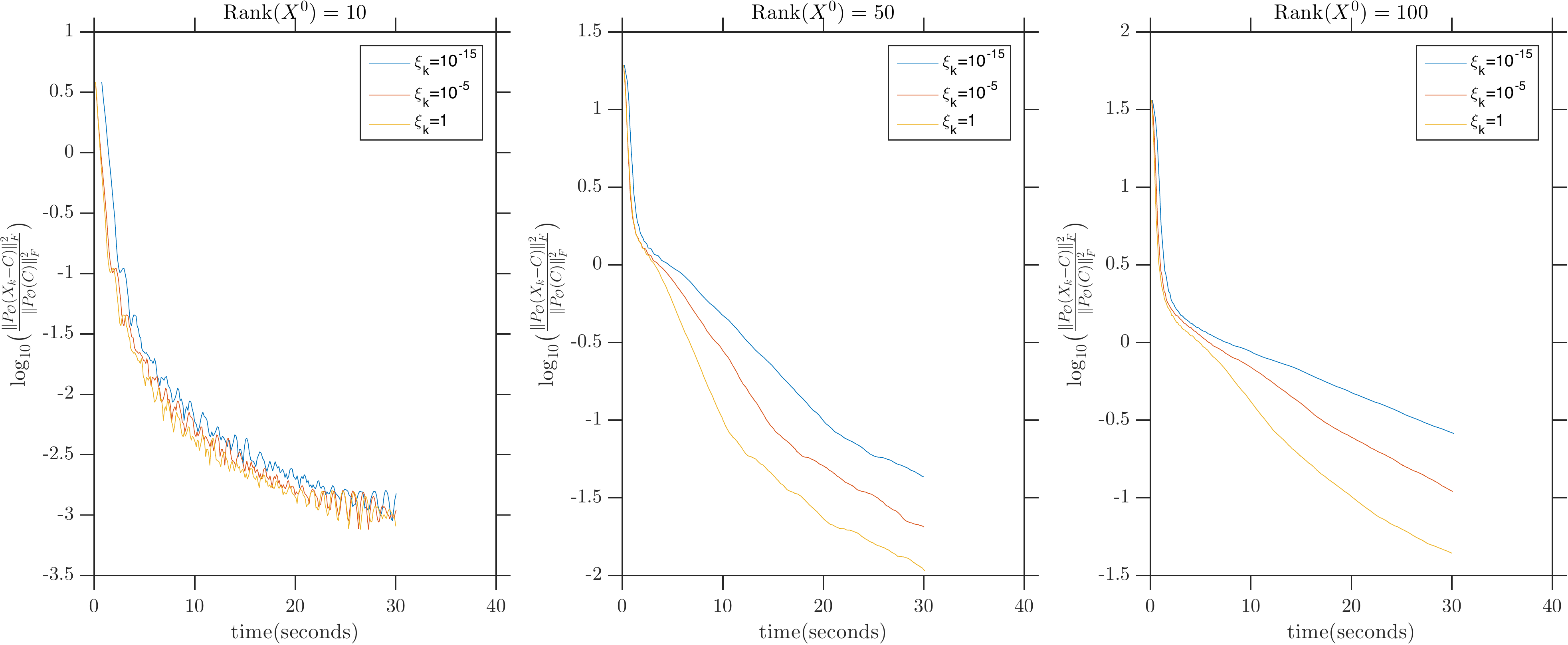}
	\caption[]{The above plots demonstrate the relation between the relative objective value $\log (\frac{\|P_{\mymathcal{O}}(X_k)-P_{\mymathcal{C}}(C)\|_F^2}{\|P_{\mymathcal{O}}(C)\|^2_F})$
		and the clock time for different combinations of rank $r = \mbox{Rank}(X^0)$ and tolerance parameters $\xi_k$.}
	\label{f0}
	\vspace*{-.5cm}
\end{figure}

One surprising feature of these graphs is the oscillation of relative error that
occurs for the model with $r=10$ once the relative error has reached $10^{-2}$ or so.
This oscillation as the algorithm approaches the optimum is due to the stepsize rule $\gamma_k = \frac{2}{k+2}$.
To see how this stepsize leads to oscillation,
suppose for simplicity that for some iterate $k_0$, $X_{k_0-1} =X^0$.
We expect this iterate to have a very low objective value; indeed, in our experiments
we found that the relative objective at $X^0$ is around $5 \times 10^{-4}$ when $r=10$.
Then in the next iteration, we add $V_{k_0}$ to $X_{k_0-1}$ with step size $\frac{2}{k_0+2}$.
Hence $X_{k_0}$ is at least $\frac{2}{k_0+2}\alpha$ away from the true solution.
This very likely will increase the relative objective since our $p$ is $0.8$.
Suppose further that $V_{k_0+1}=-V_{k_0}$. Then we almost return to $X^0$
in the next iteration and again enjoy a small relative objective.
For higher rank $X^*$, the oscillation begins at later iterations (not shown),
as the algorithm approaches the solution.

Using line search eliminates the oscillation, but increases computation time
for this problem. We do not consider linesearch further in this paper.

Our goal in this problem is not simply to find the solution of Problem \eqref{op6} but to produce a matrix $X$ close to $X^0$. Hence we also study the numerical convergence of the relative error $\|X-X^0\|_F^2/\|X^0\|_F^2$. 
Figure \ref{f1} shows experimental results on the relationship between the relative error $\frac{\|X_k-X^0\|_F^2}{\|X^0\|^2_F}$ (on a log scale) and the actual clock time
under different combinations of rank $r$ and tolerance $\xi_k$. The evolution of $\frac{\|X_k-X^0\|_F^2}{\|X^0\|^2_F}$ is very similar to the evolution of $\frac{\|P_{\mymathcal{O}}(X_k)-P_{\mymathcal{O}}(C)\|_F^2}{\|P_{\mymathcal{O}}(C)\|^2_F}$ in Figure \ref{f0}.
\begin{figure}[h]
	\vspace*{-.5cm}
	\includegraphics[width=\linewidth  ]{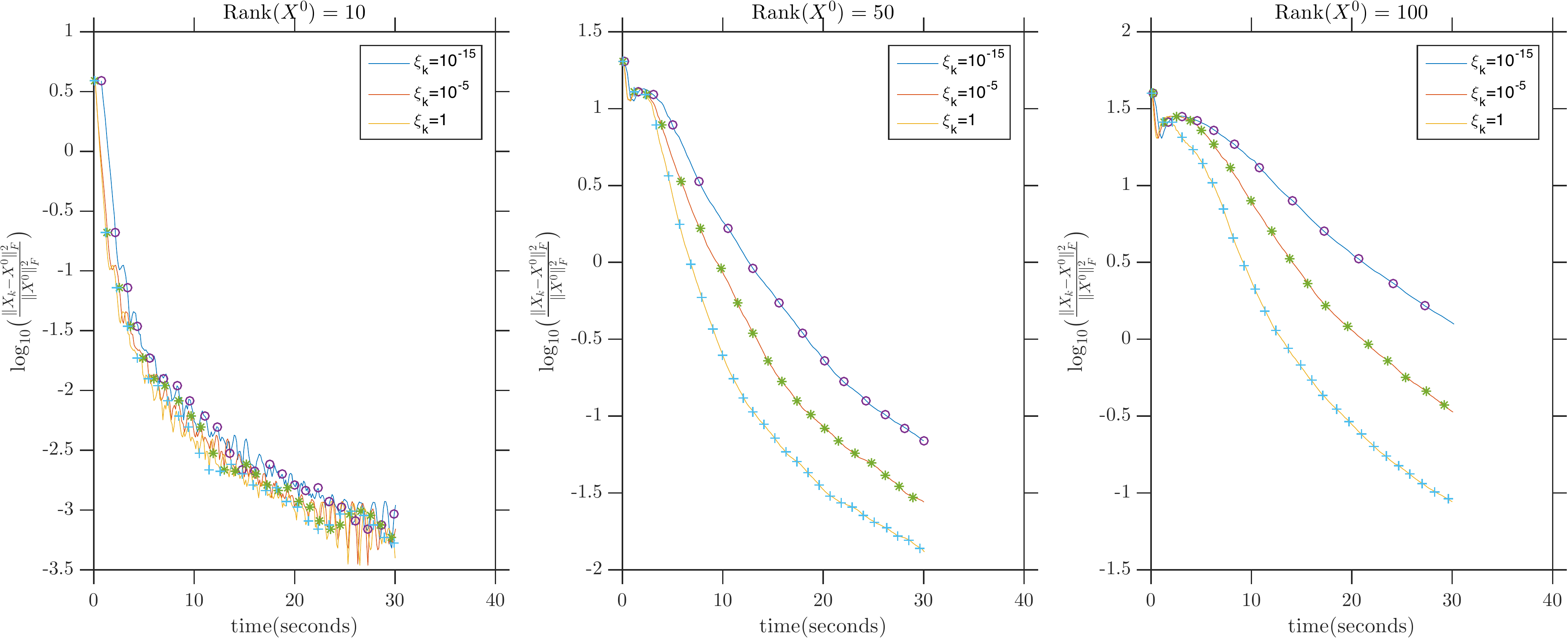}
	\caption{The above plots demonstrate the relation between the relative distance to the solution $\log (\frac{\|X_k-X^0\|_F^2}{\|X^0\|^2_F})$
		and the clock time for different combinations of rank $r = \rank(X^0)$ and tolerance parameters $\xi_k$.
		We plot a marker on the line once every ten iterations (in this figure only).}
	\label{f1}
	\vspace*{-.6cm}
\end{figure}
\begin{figure}[h]
	\includegraphics[width=\linewidth]{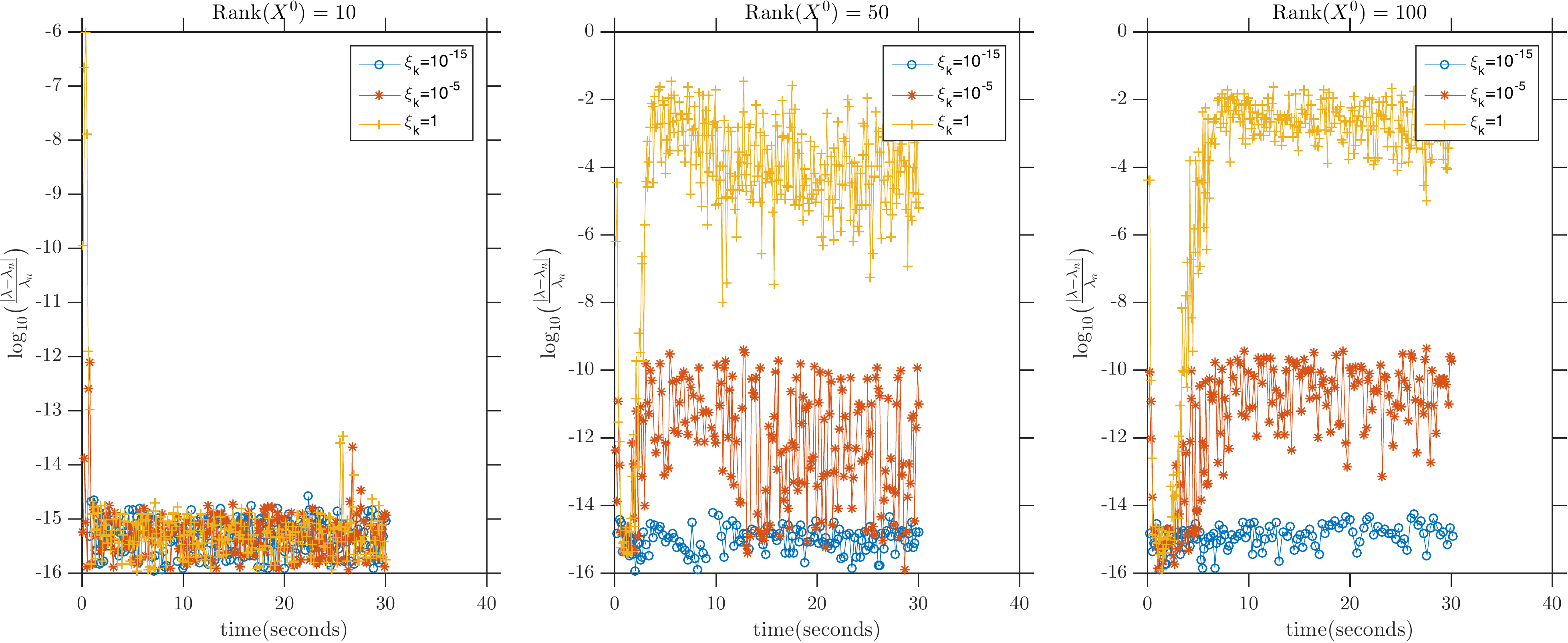}
	\caption{The vertical axis is the relative difference between approximate eigenvalue $\lambda$ with $\xi_k=\{10^{-15},10^{-5},1\}$
		and the very accurate eigenvalue $\lambda_n$ of $\lambda_n(\nabla f(X_{k-1}))$, computed with tolerance $\xi_k$ equal to machine precision $10^{-16}$.}
	\label{f2}
		\vspace{-.3cm}
\end{figure}

The assumption that the approximate eigenvalue $\lambda$ returned by the eigenvalue solver
is approximately equal to the true smallest eigenvalue $\lambda_n$ (Equation \eqref{ap1})
is supported by Figure \ref{f2}. We computed the true eigenvalue $\lambda_n$ by calling ARPACK with a very tight tolerance.
It is interesting that for a low rank model, the estimate $\lambda$ is very accurate
even if $\xi_k$ is large. The relative error in $\lambda$ is about $10^{-2}$ on average when $\xi_k=1$ for high rank models. However, this is not too large: the relative error in our iterate $\frac{\|X_k-X^0\|_F^2}{\|X^0\|^2_F}$ is also about $10^{-2}$, hence these two errors are on the same scale.

Figure \ref{f3} shows the error $\epsilon_k = v^T\nabla f(X_{k-1})v - \alpha  \lambda_{n}(\nabla f(X_{k-1}))$
achieved by our linear optimization subproblem solver. It can be seen that for a constant tolerance $\xi_k$, the error $\epsilon_k$ is also almost constant after some initial transient behavior. Hence controlling $\xi_k$ indeed controls $\epsilon_k$.
\begin{figure}[h]
	\vspace{-.5cm}
	\includegraphics[width=\linewidth]{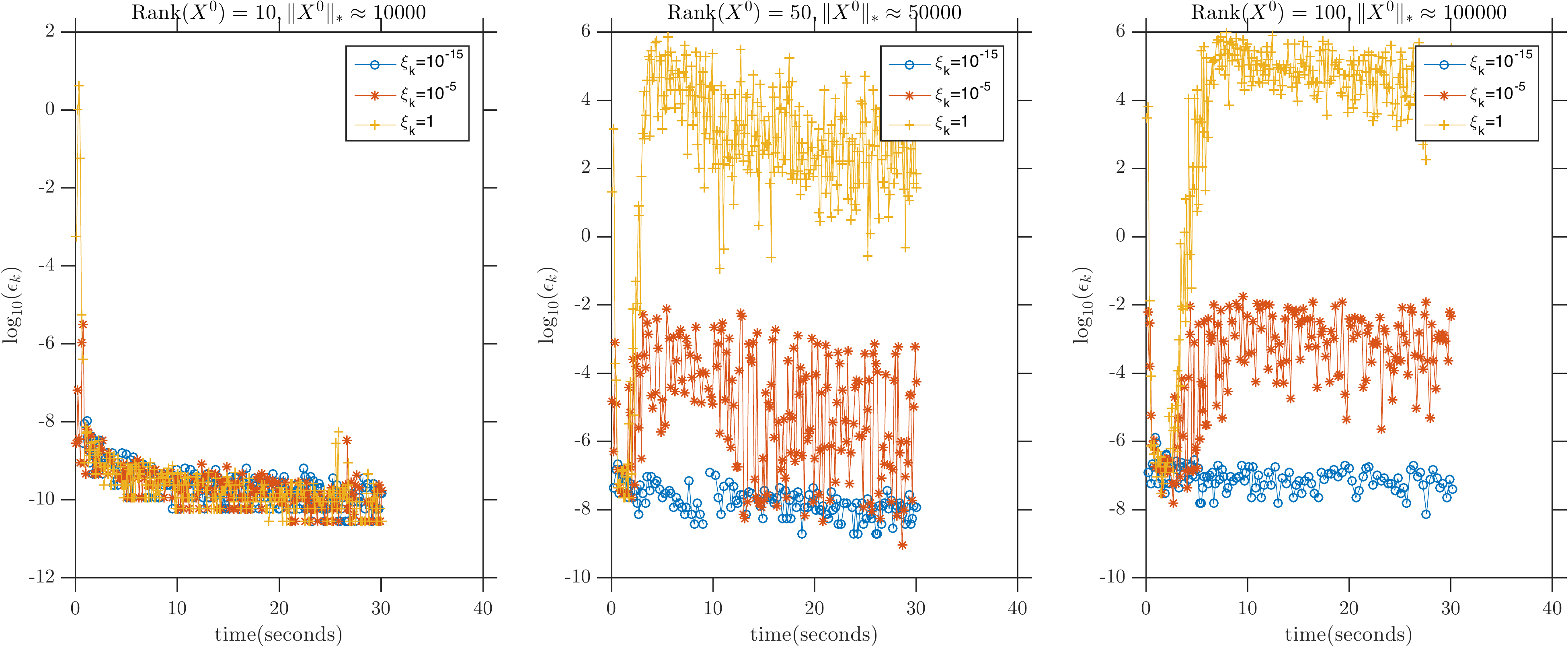}
	\caption{The actual evolution of the error $\epsilon_k=\alpha v^T\nabla f(X_{k-1})v -\alpha  \lambda_{n}(\nabla f(X_{k-1})).$}
	\label{f3}
	\vspace{-.5cm}
\end{figure}\begin{figure}[h]
\vspace{-0.5cm}
\includegraphics[width=\linewidth]{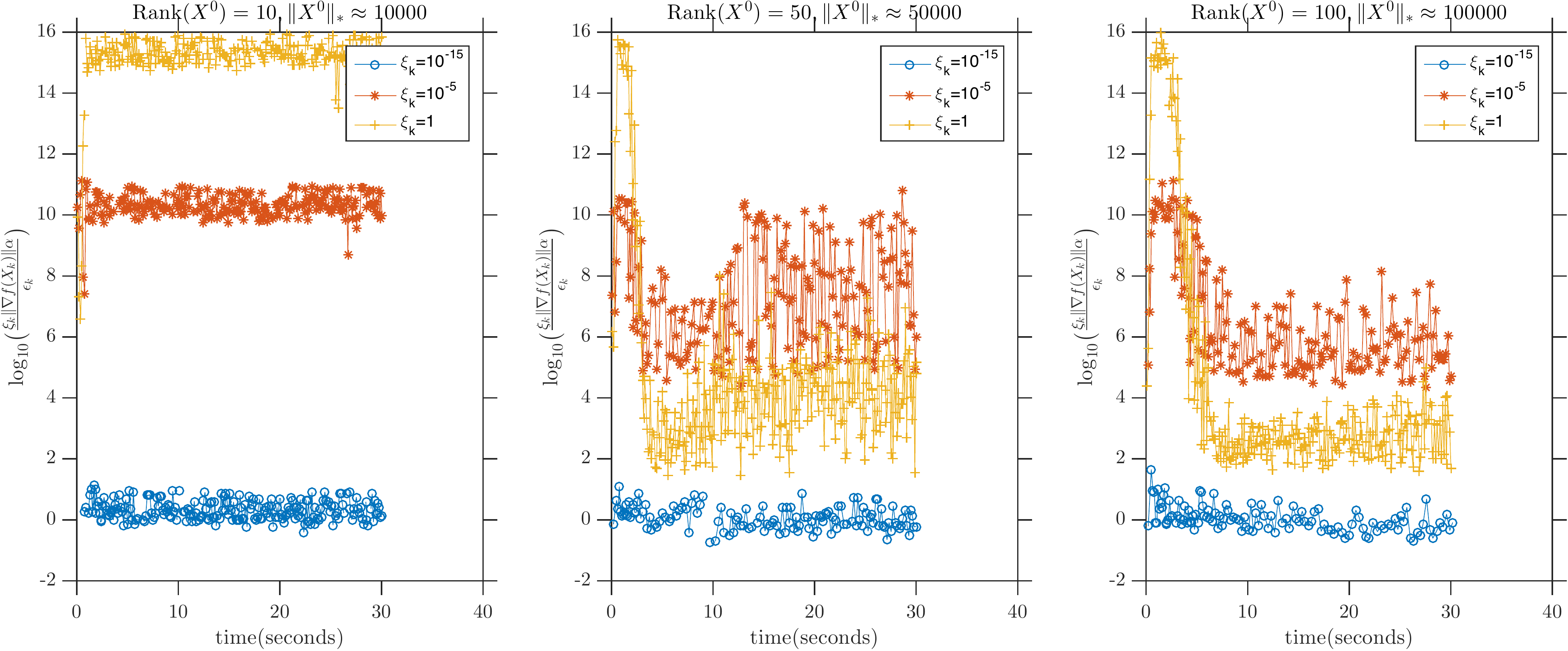}
\caption{Empirical evolution of the ratio $\frac{\xi_k \alpha \|\nabla f(X_{k-1})\| }{\epsilon_k}$.
	Recall that $\xi_k \alpha  \|\nabla f(X_{k-1})\| $ is an upper bound of $\epsilon_k$. }
\label{f4}
\end{figure}
Since our $\epsilon_k$ is approximately constant due to constant choice of $\xi_k$,
rather than decreasing as required by the assumptions of Theorem \ref{t0},
one might wonder whether the conclusion of Theorem \ref{t0} still holds.
The answer is yes. In fact, we found that at each iteration throughout our numerical experiments,
the inequality $\epsilon_k=\alpha v^T\nabla f(X_{k-1})v - \alpha \lambda_{n}(\nabla f(X_{k-1})) \leq \gamma_k LD^2 \delta$ is satisfied, with $\delta =1$ and $D = 2\|X^0\|_*$.
Thus the conclusion of Theorem \ref{t0} is still satisfied in our numerical results,
although we do not have explicit control over the error $\epsilon_k$.

We examine the accuracy of our bound $\xi_k\alpha\|\nabla f(X_{k-1})\|$ on $\epsilon_k$ in Figure \ref{f4}. It shows that our bound is rather conservative for higher value of $\xi_k$.

\section{ Stochastic Variance Reduced Frank-Wolfe (SVRF) algorithm with Approximate Oracle (\SVRF)}

Having seen that Frank-Wolfe is robust to using an approximate oracle when solving linear optimization subproblem \eqref{sp1}, we now turn to our second challenge: computing the gradient $\nabla f$.

To formalize the challenge, we will consider the optimization problem
\beq
\ba{ll}
\mbox{minimize} & f(x) :=\frac{1}{n}\sum_{i=1}^n f_i(x) \label{op2}\\
\mbox{subject to} &  x\in \Omega,\\
\ea
\eeq
where $x\in \reals^m$.  For each $i=1,\dots,n$,  $f_i$ is a convex continuously differentiable real valued function and $\Omega$ is a compact convex set in $\reals^m$. This is a particular instance of Problem \eqref{pr1}.

Problem \eqref{op2} is common in statistics and machine learning,
where each $f_i$ measures the error in observation $i$.
Computing the gradient $\nabla f$ in this setting is a challenge,
since the number of observations $n$ can be enormous.

One way to address this challenge is to compute an approximation to the gradient rather than the exact gradient.
We sample $l$ elements $i_1,\dots,i_l$ from the set $\{1,\dots,n\}$ with replacement and compute the \emph{stochastic gradient}
$$\tilde{\nabla} f(x)= \frac{1}{l}\sum_{j=1}^l \nabla f_{i_j}(x).$$
The parameter $l$ is called the size of the minibatch $\{i_1,\dots,i_l\}$.
The computational benefit here is that we compute only $l\ll n$ derivatives.
Intuitively, we expect this method to work since $\Expect[\tilde{\nabla}f(x)] = \nabla f(x)$.

Of course, the computational benefit does not come for free. This approach suffers one major drawback:
\begin{itemize}
	\item the stochastic gradient $\tilde{\nabla}f(x)$ may have very large variance $\var(\|\tilde{\nabla}f(x)\|_2)$
	even if $x$ is near $x^\star$. Large variance will destabilize any algorithm using $\tilde{\nabla}f(x)$,
	since even near the solution where $\|\nabla f(x)\|$ is small, $\|\tilde{\nabla}f(x)\|$ may be large.
\end{itemize}
One simple way to ensure  that $\tilde{\nabla}f(x)$ concentrates near $\nabla f(x)$ is to
increase the minibatch size $l$ as $\var(\tilde{\nabla}f(x)) = \frac{1}{l}\var(\nabla f_i(x))$,
where $i$ is chosen uniformly from $\{1,\dots,n\}$. 
But using a very large minibatch size $l$ defeats the purpose of using a stochastic gradient.

Variance reduction techniques endeavor to avoid this tradeoff \cite{svrg}.
Instead of using a large minibatch at each iteration,
they occasionally compute a full gradient and use it to reduce the variance of $\tilde{\nabla}f(x)$.
The modified stochastic gradient is called the variance-reduced stochastic gradient.
Johnson and Zhang \cite{svrg} introduced one way to perform variance reduction.
Specifically, they define a variance-reduced stochastic gradient at a point $x\in \Omega$
with respect to some \textit{snapshot} $x_0\in \Omega$ as
\[
\tilde{\nabla}f(x;x_0)=\nabla f_i(x) -(\nabla f_i(x_0)-\nabla f(x_0)),
\]
where $i$ is sampled uniformly from $\{1,\dots,n\}$.
Notice we require the full gradient $\nabla f(x_0)$ at the snapshot, but only the
gradient of the $i$th function $\nabla f_i(x)$ at the point $x$.
In this case, we still have  $\Expect \tilde{\nabla} f(x;x_0) =\nabla f(x)$, and the variance is
\[
\var(\|\tilde{\nabla}f(x;x_0)\|_2) =
\frac{1}{n} \sum_{i=1}^n \| \nabla f_i(x) -\nabla f_i(x_0)+(\nabla f(x_0)-\nabla f(x))\|^2_2.
\]
If $x$ and $x_0$ are near $x^\star$, the variance will be near zero and so indeed the variance is reduced.
We can further reduce the variance using a minibatch by independently sampling $l$
variance-reduced gradients $\tilde{\nabla}f(x;x_0)$ and taking their average.

Hazan and Luo \cite[Theorem 1]{Hazan} introduced
the stochastic variance reduced Frank-Wolfe (SVRF) algorithm,
which augments the Frank-Wolfe algorithm with the variance reduction technique of Johnson and Zhang,
and showed that it converges in expectation when an exact oracle is used for the linear optimization subproblem \eqref{sp1}.
As we will see in Theorem \ref{t1}, the number of evaluation of full gradient and stochastic gradient is also considerably small.

As we saw in the previous section, Frank-Wolfe with an approximate oracle converges
at the same rate as the one using an exact oracle
for the linear optimization subproblem \eqref{sp1}.
One naturally wonders whether an approximate oracle is allowed when we use stochastic gradients.
We will show below that the resulting algorithm,
which we call SVRF with approximate oracle
(\SVRF, pronounced as ``tilde SVRF'') and
present as Algorithm \ref{svrf}, indeed works well.
Note that when $\epsilon_k=0$ for each $k$, Algorithm \ref{svrf} reduces to SVRF. 

\begin{algorithm}
	\caption{SVRF with approximate oracle (\SVRF) \label{svrf}}
	\begin{algorithmic}[1]
		\STATE {\bf Input:} Objective function $f = \frac{1}{n} \sum_{i=1}^n f_i$
		\STATE{\bf Input:} A feasible starting point $w_{-1}\in \Omega$
		\STATE{\bf Input:} Stepsize $\gamma_k$, minibatch size $m_k$, epoch length $N_t$ and tolerance sequence $\epsilon_k$
		\STATE {\bf Initialize:} Find $x_0$ s.t. $ \nabla f(w_{-1})^Tx_0 \leq \min_{x\in \Omega} \nabla f(w_{-1}) ^Tx+\epsilon_0$.
		\FOR{$t=1,2, \ldots, T$}
		\STATE Take a snapshot $w_0 = x_{t-1}$ and compute gradient $\nabla f(w_0)$.
		\FOR{$k=1$ to $N_t$} \label{kiter}
		\STATE Compute $g_k$, the average of $m_k$ iid samples of $\tilde{\nabla}f(w_{k-1},w_0)$.
		\STATE Compute $v_k$ s.t. $g_k^T v_k \leq \min_{v\in \Omega} g_k^Tv +\epsilon_k$.
		\STATE Update $w_k := (1-\gamma_k)w_{k-1}+\gamma_k v_k$. \label{lst:line:up}
		\ENDFOR
		\STATE Set $x_t = w_{N_t}$ \label{kiter2} .
		\ENDFOR

		\STATE {\bf Output:} The last iteration result $x_T$.
	\end{algorithmic}

\end{algorithm}

We give a quantitative description of the objective value convergence $f(x_k) -f(x^\star)$
for Algorithm \ref{svrf} in Theorem \ref{t1}.
Moreover, we show that the convergence rate is the same as the one using the exact
subproblem oracle up to a multiplicative user-specified constant.

In Algorithm \ref{svrf}, each time we take a snapshot, we let $k=1$ again
and the algorithm essentially restarts. Another option available is not to restart $k$.
This modification is suggested and implemented in \cite{Hazan};
further, they observe this algorithmic variant is more stable.
This modification ensures that the stepsize always decreases, and so intuitively should increase the stability.

We state this modification as Algorithm \ref{svrf2} below.
We show it converges in expectation with the same rate as Algorithm \ref{svrf},
and that it converges almost surely.
These results are new to the best of our knowledge,
and theoretically justify why a diminishing stepsize makes the algorithm
more stable: the optimality gap converges almost surely to $0$ rather than just in expectation!

\begin{algorithm}
	\caption{Stable \SVRF, $k$ increasing in  line~\ref{kiter} of Algorithm \ref{svrf} }\label{svrf2}
	\begin{algorithmic}[1]
		\STATE \dots as Algorithm \ref{svrf}, except replacing the chunk from line~\ref{kiter}  to line \ref{kiter2} with the following chunk and start $k$ at $k=1$ when $t=1$.
		\WHILE{$k \leq N_t$}
		\STATE Compute ${g}_k$, the average of $m_k$ iid samples of     $\tilde{\nabla}f(w_{k-1},w_0)$.
		\STATE Compute $v_k$ s.t. $g_k^T v_k \leq \min_{v\in \Omega} g_k^Tv +\epsilon_{k}.$
		\STATE Update $w_k := (1-\gamma_k)w_{k-1}+\gamma_k v_k$ and $k=k+1$.
		\ENDWHILE
		\STATE Set $x_t = w_{N_t}$.
	\end{algorithmic}
\end{algorithm}

\section{Theoretical guarantees for \SVRF}

We show below that \SVRF has the same convergence rate as SVRF, up to constants
depending on the error level $\delta$.
The proof is analogous to the one in Hazan and Luo \cite[Theorem 1]{Hazan},
with some additional care in handling the error term.

\begin{theorem}\label{t1}
	Suppose each $f_i$ is $L$-smooth and $\Omega$ has diameter $D$.
	Then for any $\delta>0$,
	Algorithms \ref{svrf} and \ref{svrf2}
	with parameters
	\[
	\gamma_k =\frac{2}{k+1},
	\quad m_k=96(k+1), \quad N_t =2^{t+3}-2,
	\quad \epsilon_{k}=\frac{LD^2}{2}\gamma_k\delta
	\]
	ensure that for any $t$,
	\[
	\Expect[ f(x_t)-f(x^\star)]\leq \frac{LD^2(1+\delta)}{2^{t+1}}.
	\]
	Moreover, for any $k$,
	\[
	\Expect [f(w_k)-f(x^\star)] \leq \frac{4LD^2(1+\delta)}{k+2}.
	\]
\end{theorem}

One might be concerned that \SVRF is impractical, since the minibatch size required
to compute the approximate gradient increases linearly with $k$.
However, when the number of terms $n$ in the objective is sufficiently large,
in fact the complexity of \SVRF is lower than that of
Algorithm \ref{fwap}, Frank-Wolfe with approximate oracle.
Under the parameter settings in Theorem \ref{t1}, with a bit extra work, we see that
\SVRF requires
$\mymathcal{O}(\ln (\frac{LD^2(1+\delta)}{\epsilon}))$ full gradient evaluations,
$\mymathcal{O}(\frac{L^2D^4(1+\delta)^2}{\epsilon^2})$ stochastic gradient evaluations,
and the solution of $\mymathcal{O}(\frac{LD^2(1+\delta )}{\epsilon})$ linear optimization subproblems.
As a comparison, Algorithm \ref{fwap}, Frank-Wolfe with approximate oracle,
under the parameter settings in Theorem \ref{t0}, requires
$\mymathcal{O}(\frac{LD^2(1+\delta)}{\epsilon})$ full gradient evaluations and
the solution of the same number of linear optimization subproblems.
Suppose that the cost of computing the full gradient is $n$ times the cost of computing one stochastic gradient.
Then \SVRF enjoys a smaller computational cost than Algorithm \ref{fwap} if
$$\mymathcal{O}\left(\ln (\frac{LD^2(1+\delta)}{\epsilon})\right)
+ \frac{1}{n} \mymathcal{O}\left(\frac{L^2D^4(1+\delta)^2}{\epsilon^2}\right)
< \mymathcal{O}\left(\frac{LD^2(1+\delta)}{\epsilon}\right),$$
which is satisfied for large $n$.

We begin the proof using the smoothness of $f_i$ \cite[Theorem 2.1.5]{nesterov2013introductory}.

\begin{proposition} \label {p1}
	Suppose a real valued function $g$ is convex and $L$-smooth over its domain $\reals^n$.
	Then $g$ satisfies
	\begin{align*}
	\| \nabla g(w)-\nabla g(v)\|^2 \leq 2L(g(w)-g(v)-\nabla g(v)^T(w-v))
	\end{align*}
	for all $w,v \in \reals^n$.
\end{proposition}

\begin{proof}
	Consider $h(w) = g(w)-\nabla g(v) ^Tw$, which is also convex and $L$-smooth.
	The minimum of $h(w)$ occurs at $w = v$, since $\nabla h(v) =0$.
	Hence
	\beq
	\ba{ll}
	h(v)- h(w) &\leq h(w-\frac{1}{L}\nabla h(w))- h(w)\\
	& \leq- (\nabla g(w) -\nabla g(v))^T(\frac{1}{L}(\nabla g(w)-\nabla g(v))) +\frac{L}{2} \frac{1}{L^2} \|\nabla g(w) - \nabla g(v)\|^2 \\
	& \leq -\frac{1}{2L}\|\nabla g(w) - \nabla g(v)\|^2
	\ea
	\eeq
	where the second inequality is due to the smoothness of $h$. Substitute $h(w) = g(w)- \nabla g(v)^Tw$ back into the above inequality gives Proposition \ref{p1}.
\end{proof}

The second ingredient of the proof is bounding the variance of the reduced variance gradient $\tilde{\nabla} f(x_0,x)$  in terms of the difference between the current value and the optimal function value. Note that  $\tilde{\nabla} f(x_0,x)$ is an unbiased estimator of $\nabla f(x)$. The proof relies on Proposition \ref{p1} and can found in  Hazan and Luo \cite[Lemma 1]{Hazan}.
\begin{lemma}\label{l1}
	For any $x,x_0 \in \Omega$, we have
	\begin{align*}
	\Expect [ \| \tilde {\nabla}f(x;x_0)-\nabla f(x)\|^2] \leq 6L(2\Expect[ f(x)-f(x^\star)] +\Expect[f(x_0)-f(x^\star)]).
	\end{align*}
\end{lemma}

\begin{proof}
	\beq
	\ba{lll}
	\Expect[ \|\tilde{\nabla} f(x;x_0)-\nabla f(x)\|^2]
	&= &\Expect [\| \nabla f_i(x) -\nabla f_i(x_0) +\nabla f(x_0) - \nabla f(x)\|^2]\\
	 &=&\Expect[\bigr\| \bigr(\nabla f_i(x)-\nabla f_i(x^\star) \bigr ) - \bigr (\nabla f_i(x_0)-\nabla f_i(x^\star)\bigr ) \\& &+ \bigr(\nabla f(x_0)- \nabla f(x^\star)\bigr) -\bigr(\nabla f(x)-\nabla f(x^\star)\bigr)\bigr\|^2] \\
	&\leq  &3\Expect[\| \nabla f_i(x)-\nabla f_i(x^\star)\|^2+\bigr\| \bigr(\nabla f_i(x)-\nabla f_i(x^\star)\bigr) \\& &- \bigr(\nabla f(x_0)- \nabla f(x^\star)\bigr)\bigr\|^2 +\|\nabla f(x)-\nabla f(x^\star)\|^2]\\
	&\leq &3 \Expect[\| \nabla f_i(x)-\nabla f_i(x^\star) \|^2+\| \nabla f_i(x_0)-\nabla f_i(x^\star) \|^2\\& &+\|\nabla f(x)-\nabla f(x^\star)\|^2]
	\ea
	\eeq
	where the first inequality is due to Cauchy-Schwarz and the fact that $2ab \leq a^2 +b^2$ for any $a,b\in \reals$. The second inequality is the variance $\Expect [\| (\nabla f_i(x)-\nabla f_i(x^\star)) - (\nabla f(x_0)- \nabla f(x^\star))\|^2]$ is less than its second moment $\Expect [\| (\nabla f_i(x)-\nabla f_i(x^\star)) \|^2]$.

	Now we apply Proposition \ref{p1} to the three terms above. For example, for the first term, we have
	\beq
	\ba {ll}
	\Expect [\|\nabla f_i(x)- \nabla f_i(x^\star)\|^2]& \leq 2L \Expect [f_i(x)-f_i(x^\star)- \nabla f_i(x^\star)^T(w-w^\star)]\\ &= 2L  (f(x)-f(x^\star)- \nabla f(x^\star)^T(x-x^\star))\\
	 &\leq 2L(f(x)-f(x^\star))
	\ea
	\eeq
	where the second inequality is due to the optimality of $x^\star$. Applying the proposition similarly to other two terms yields the lemma.
\end{proof}

The key to the proof of Theorem  \ref{t1} is the following lemma.

\begin{lemma}\label{l2}
	For any $t$ and $k$ in Algorithm \ref{svrf} and \ref{svrf2}, we have
	\[\Expect [f(w_k)-f(x^\star)] \leq \frac{4LD^2(1+\delta)}{k+2}\]
	if \[ \Expect [\| g_s -\nabla f(w_{s-1})\|^2]\leq \frac{L^2D^2(1+\delta)^2}{(s+1)^2}\]
	for all $s\leq k$.
\end{lemma}

\begin{proof}
	The $L$-smoothness of $f$ gives that for any $s\leq k$, \[ f(w_s) \leq f(w_{s-1})+\nabla f(w_{s-1})^T(w_s-w_{s-1})+\frac{L}{2}\|w_s-w_{s-1}\|^2 .\]

	Under Algorithm \ref{svrf} or \ref{svrf2}, we have $w_s= (1-\gamma_s)w_{s-1} +\gamma_s v_s $. Plugging this in the above inequality gives  \[f(w_s) \leq  f(w_{s-1}) +\gamma_s \nabla f(w_{s-1})^T(v_s-w_{s-1}) +\frac{L\gamma_s^2}{2} \|v_s-w_{s-1}\|^2.\]

	Using the definition of the diameter of $\Omega$, we can rearrange the previous inequality as
	\[ f(w_s) \leq f(w_{s-1})+\gamma_s g_s^T(v_s-w_{s-1})+\gamma_s(\nabla f(w_{s-1})-g_s)^T(v_s-w_{s-1})+\frac{LD^2\gamma_s^2}{2}.\]

	Since $g_s^Tv_s \leq \min_{w\in \Omega } g_s^Tw +\frac{\gamma_s\delta LD^2}{2}\leq g_s^Tx^\star +\frac{\gamma_s\delta LD^2}{2}$,
	we arrive at
	\begin{align}
	f(w_s) \leq & f(w_{s-1})+\gamma_s \nabla f(w_{s-1})^T(x^\star-w_{s-1}) \label{ineq1}\\
	&+\gamma_s(\nabla f(w_{s-1})-g_s)^T(v_s-x^\star)+\frac{LD^2(1+\delta)\gamma_s^2}{2} \nonumber.
	\end{align}

	By convexity, the term $\nabla f(w_{s-1})^T(x^\star-w_{s-1})$ is upper bounded by $f(x^\star) - f(w_{s-1})$, and Cauchy-Schwarz inequality yields that $$|(\nabla f(w_{s-1})-g_s)^T(v_s-x^\star)| \leq D\|g_s -\nabla f(w_{s-1})\|.$$ The assumption on $\|g_s -\nabla f(w_{s-1})\|^2$ gives $\Expect [\|g_s -\nabla f(w_{s-1})\|]$ is at most $\frac{LD(1+\delta)}{s+1}$ by Jensen's inequality. Recalling $\gamma_s =\frac{2}{s+1}$, we have
	\begin{align*}
	& \Expect[ f(w_s)-f(x^\star)]\\
	\leq & (1-\gamma_s)\Expect [f(w_{s-1})-f(x^\star)] +\frac{LD^2 \gamma_s^2(1+\delta)}{2} +\frac{LD^2\gamma_s^2(1+\delta)}{2}\\
	=& (1-\gamma_s)\Expect[f(w_{s-1})-f(x^\star)] +LD^2\gamma_s^2(1+\delta).
	\end{align*}

	We now prove $\Expect [f(w_k)-f(x^\star)]\leq \frac{4LD^2(1+\delta)}{k+2}$ by induction. The base case $k=1$ is simple by noting $\gamma_1=1$ and $$\Expect[ f(w_1)-f(x^\star)]
	\leq
	(1-\gamma_1)\Expect[f(w_{0})-f(w^*)] +\gamma_1 LD^2(1+\delta)=LD^2(1+\delta).$$
	Now suppose for $k=s-1$, $\Expect[f(w_{s-1})-f(x^\star)]\leq \frac{4LD^2(1+\delta)}{s+1} $.
	Then with $\gamma_ s =\frac{2}{s+1}$, we have for $k=s$
	\[
	\Expect [f(w_s)-f(x^\star)]\leq \frac{4LD^2(1+\delta)}{s+1}\left(1-\frac{2}{s+1}+\frac{1}{s+1}\right)\leq \frac{4LD^2(1+\delta)}{s+2},\]
	which completes the induction.
\end{proof}

With this lemma, we are able to prove Theorem \ref{t1}

\begin{proof}[Proof of Theorem \ref{t1}]
	We proceed by induction. In the base case $t=0$, we have
	\begin{align*}
	f(x_0) & \leq f(w_{-1}) +\nabla f(w_{-1})^T(x_0-w_{-1}) +\frac{L}{2}\|w_{-1}-x_0\|^2 \\
	&\leq f(w_{-1})+\nabla f(w_{-1})^T(x^\star-w_{-1}) +\frac{LD^2}{2} +\frac{LD^2\delta}{2}\\
	&\leq f(x^\star) +\frac{LD^2(1+\delta)}{2},
	\end{align*}
	where we use the $L$-smoothness in the first inequality, the near optimality of $x_0$ in the second inequality and convexity of $f$ in the last inequality.

	Now we assume that $\Expect [f(x_{t-1}) -f(x^\star)] \leq \frac{LD^2(1+\delta)}{2^t}$ and we are in Algorithm \ref{svrf}. We consider iteration of the algorithm and use another induction to show $\Expect [f(w_k)-f(x^\star)]\leq \frac{4LD^2(1+\delta)}{k+1}$ for any $ k\leq N_t$. The base case $w_0 =x_{t-1}$ is clearly satisfied because of the induction hypothesis $\Expect [f(x_{t-1} )-f(x^\star)] \leq \frac{LD^2(1+\delta)}{2^t}$.
	Given the induction hypothesis $\Expect [f(w_{s-1}-f(x^\star))]\leq \frac{4LD^2(1+\delta)}{s+1}$ for any $s\leq k$, we have
	\begin{align*}
	&\Expect[\|g_s -\nabla f(w_{s-1})\|^2]\\
	\leq & \frac{6L}{m_s}(2\Expect [f(w_{s-1})-f(x^\star)]+\Expect[f(w_0)-f(x^\star)])\\
	\leq & \frac{6L}{m_s}\left( \frac{8LD^2(1+\delta)}{s+1} + \frac{LD^2(1+\delta)}{2^t}\right)\\
	\leq  & \frac{6L}{m_s}\left(\frac{8LD^2(1+\delta)}{s+1} + \frac{8LD^2(1+\delta)}{s+1}\right) \\
	=& \frac{L^2D^2(1+\delta)}{(s+1)^2}\leq \frac{L^2D^2(1+\delta)^2}{(s+1)^2}
	\end{align*}
	where the first inequality use Lemma \ref{l1} and the fact that variance reduced by a factor $m_s$ as $g_s$ is the average of $m_s$ iid samples of $\tilde{\nabla} f(w_{s-1};w_0)$ and the second and third inequality are due to the two induction hypothesis and $s\leq N_t =2^{t+3}-2$. The last equality is due to the choice of $m_s$. Therefore, we see the condition of Lemma \ref{l2} is satisfied and the induction is completed.

	Now suppose we are in the situation of Algorithm \ref{svrf2}. The only difference here is that we don't restart $k$ at $1$. Assuming that $\Expect [f(x_{s-1}) -f(x^\star)] \leq \frac{LD^2(1+\delta)}{2^t}$ for all $s\leq t-1$ and by inspecting previous argument, we only need to show $\Expect [f(w_k)-f(x^\star)]\leq \frac{4LD^2(1+\delta)}{k+1}$ for any $ k\leq N_t$. Since our $k$ is always increasing, we cannot directly employ our previous argument. By the structure of our algorithm,
	we can split the range of $k$ into $t$ cycles $\{1,\dots,N_1\}, \{N_{1}+1,\dots, N_{2}\},\dots, \{N_{t-1}+1,\dots, N_t\}$.
	Now within each cycle, we can apply the previous argument, and thus we indeed have $\Expect [f(w_k)-f(x^\star)]\leq \frac{4LD^2(1+\delta)}{k+2}$ for any $ k\leq N_t$.

	By the choice of $N_t$, we see
	\begin{align*}
  \Expect [f(w_{N_t})-f(x^\star)]
	&= \Expect[f(x_{t})-f(x^\star)]\\
	&\leq  \frac{4LD^2(1+\delta)}{N_t+2} \\
&	=  \frac{LD^2(1+\delta)}{2^{t+1}}.
	\end{align*}

\end{proof}
The authors of \cite{Hazan} mention that Algorithm \ref{svrf2} seems to be more stable than Algorithm \ref{svrf}. We give the following theoretical justification for this empirical observation.

\begin{theorem} \label{t2}
	Under the same assumption of Theorem \ref{t1},  we have
	\[\lim_{s\rightarrow \infty}f(w_s) =f(x^\star)\]
	with probability $1$ for Algorithm \ref{svrf2}.
\end{theorem}

The theorem asserts that the objective value will converge to the true optimum under almost any realization while Theorem \ref{t1} tells we have convergence in expectation.

The proof relies on the martingale convergence theorem, which we recall here.

\begin{theorem}[Martingale convergence theorem]
	Let $\{X_t\}_{t=1}^n$ be a sequence of real random variables and $\Expect_s$ to be the conditional expectation conditional on all $X_i, i\leq s-1$, then if $X_t$ is a supermartingale, \ie,
\begin{equation*}
\begin{array}{l}
\Expect_s (X_s) \leq X_{s-1}
\end{array}
\end{equation*}
and for all $t$,
	$$X_t \geq L$$ for some $L$. Then there is a random variable $X$ that
	$$X_s \rightarrow X \quad \text{almost surely.}$$
\end{theorem}
To make the presentation clear, we first prove a simple lemma in constructing a martingale.

\begin{lemma} \label{l4}
	Suppose a sequence of random variables $\{X_s\}^\infty_{s=1}$ and a deterministic sequence $\{b_s\}_{s=1}^\infty$ satisfy $\Expect_{s}(X_s) \leq X_{s-1} +b_s$ and $X_s \geq L$ for some $L\in \reals$ for all $s$ with probability $1$. Furthermore, assume that $\sum_{s=1}^\infty b_s = C < \infty$. Then $X_s + a_s$ where $a_s = C - \sum_{i=1}^s b_s$ is a supermartingale.
\end{lemma}

\begin{proof}
	The condition $X_s\geq 0$ is mainly used so that all our expectations make sense.
	We need to show $\Expect _{s}  (X_s +a_s )\leq X_{s-1} +a_{s-1}.$ Now by moving $a_s$ to the RHS and use the definition of $a_s$, we see this inequality holds because of the assumption  $\Expect_{s}(X_s) \leq X_{s-1} +b_s$.
\end{proof}

We now prove Theorem \ref{t2}.
\begin{proof}[Proof of Theorem \ref{t2}]
	Recall that $\Expect_s$ denotes the conditional expectation given all the past except the realization of $s$. Using inequality \eqref{ineq1}, we see that
	\begin{equation*}
	\begin{array}{lll}
	\Expect_s f(w_s) & \leq& f(w_{s-1})+\gamma_s \nabla f(w_{s-1})^T(x^\star-w_{s-1}) \\
	 & &+\gamma_s\Expect_s[(\nabla f(w_{s-1})-g_s)^T(v_s-x^\star)]+\frac{LD^2(1+\delta)\gamma_s^2}{2} \nonumber.
	\end{array}
\end{equation*}
	By convexity, the term $\nabla f(w_{s-1})^T(x^\star-w_{s-1})$ is upper bounded by $f(x^\star) - f(w_{s-1})$, and Cauchy-Schwarz inequality yields that $|(\nabla f(w_{s-1})-g_s)^T(v_s-x^\star)| \leq D\|g_s -\nabla f(w_{s-1})\|$.  Since $\Expect_s[\|g_s -\nabla f(w_{s-1})\|]\leq \sqrt{ \Expect_s( \|g_s -\nabla f(w_{s-1})\|^2)}$ by Jensen's inequality. Using lemma \ref{l1}, we see that
	\begin{equation*}
	\begin{array}{ll}
\sqrt{\Expect_s[\|g_s -\nabla f(w_{s-1})\|^2]}
	&\leq  \sqrt{\frac{6L}{m_s}(2\Expect_s [f(w_{s-1})-f(x^\star)]+\Expect_s[f(w_0)-f(x^\star)])}\\
	&\leq \sqrt{ \frac{18LB}{m_s}}
	\end{array}
	\end{equation*}
	where $B = \sup_{x\in \Omega} f(x)-f(x^\star)$ which is finite as $\Omega$ is compact.

	Using all previous inequalities, we see that
	\begin{equation*}
	\begin{array}{ll}
		\Expect_s [f(w_s) -f(x^\star)]
	&\leq (1- \gamma_s) (f(w_{s-1}) -f(x^\star)) + \frac{LD^2(1+\delta)\gamma_s^2}{2} + \gamma_s\sqrt{\frac{18LB}{m_s}}				\\
	&\leq f(w_{s-1}) -f(x^\star) + \frac{LD^2(1+\delta)\gamma_s^2}{2} + \gamma_s\sqrt{\frac{18LB}{m_s}}
	\end{array}
	\end{equation*}

	Since by our choice of $\gamma_s$ and $m_s$, we know that by letting $b_s =  \frac{LD^2\gamma_s^2}{2} + \gamma_s\sqrt{\frac{18LB}{m_s}}, X_s = f(w_s)- f(x^\star) \geq 0, a_s = \sum_{i=s+1}^\infty b_s$, the condition of Lemma \ref{l4} is satisfied and thus $X_s + a_s$ is indeed a super martingale.

	Now using the martingale convergence theorem, we know that $X_s +a_s$ converges to a certain random variable $X$. Since $a_s \rightarrow 0$ as $s\rightarrow \infty$, $X_s \rightarrow X$ almost surely. $X_s \geq 0$ then implies $X\geq 0$.
	But $\Expect X \leq \Expect X_s +a _s$ for any $s$ by the supermartingale property.
	Because $\Expect X_s \rightarrow 0$  by Theorem \ref{t1} and $a_s\rightarrow 0$,
	$\Expect X \leq \Expect X_s +a _s$ implies $\Expect X \leq 0$. Combine the fact $X\geq 0$ as we just argued, we see $X=0$.
	This shows that $X_s \rightarrow 0$ almost surely which is what we need to prove.
\end{proof}

The reason that the above argument does not work for Algorithm \ref{svrf} is that once in a while we restart $k$ and the sequence $b_s$ we used above will be abandoned. More precisely, since the martingale convergence theorem does not tell when the sequence is about to converge, within $t$th cycle of $k\in \{N_t+1,\dots, N_{t+1}\}$, we don't know whether the sequence $f(w_s)$ has converged or not. When we enter a new cycle, we start fresh from $k=1$ with a new $b_s$.
By contrast, for Algorithm \ref{svrf2}, we know that $k$ is always increasing and we have only one sequence $b_s$.
This observation explains why Algorithm \ref{svrf2} is likely to be more stable.

\section{\SSVRF}

In previous sections, we have seen how to augment the standard Frank Wolfe algorithm with
\begin{itemize}
	\item an approximate oracle for linear optimization subproblem \eqref{sp1},
	\item stochastic variance reduced gradients.
\end{itemize}
Now we turn our attention to the third challenge we raised in the introduction,
restricting our attention to the case where the decision variable $X \in \reals^{m \times n}$ is a matrix:
what if storing the decision variable $X$ is also costly?

Of course, if the decision variable at the solution has no structure, there is no hope to store it more cheaply:
in general, $m\times n$ space is required simply to output the solution to the problem.
However, in many settings $X$ at the solution may enjoy a low rank structure:
$X$ at the solution can be well approximated by a low rank matrix.

The idea introduced in \cite{sketch} is designed to capture this low rank structure.
It forms a linear sketch of the column and row spaces of the decision variable $X$,
and then uses the sketched column and row spaces to recover the decision variable.
The recovered decision variable approximates the original $X$ well if a low rank structure is present.

The advantage of this procedure in the context of optimization
is that the decision variable $X$ may not be low rank
at every iteration of the algorithm. However, so long as
the \emph{solution} is (approximately) low rank, we can use this procedure to
sketch the decision variable and to recover the solution from this sketch,
as introduced in \cite{sketchy}.
Notably, we need not store the entire decision variable at each iteration, but
only the sketch. Hence the memory requirements of the algorithm are substantially reduced.

Specifically, the sketch proposed in \cite{sketch} is as follows.
To sketch a matrix $X \in \reals^{m\times n}$, draw two matrices with independent normal  entries
$\Psi \in \reals^ {n \times k}$ and $\Phi \in \reals ^{l\times m}$.
We use $Y^C$ and $Y^R$ to capture the column space  and the row space  of $X$:
\begin{align}
Y^C = X\Psi \in \reals^{m \times k},\qquad  Y^R =\Phi X \in \reals^{l\times n} . \label{sk1}
\end{align}
In the optimization setting of matrix completion with Algorithm \ref{fw1}, we do not observe the matrix $X$ directly.
Rather, we observe a stream of rank one updates
\[ X \leftarrow \beta_1 X+ \beta_2 uv^T,\] where $\beta_1,\beta_2$ are real scalars.
In this setting, $Y^C$ and $Y^R$ can be updated as
\begin{align}
Y^C \leftarrow \beta_1Y^C+ \beta_2 u v^T \Psi  \in \reals^{m \times k },\quad  Y^R \leftarrow  \beta_1Y^R + \beta_2 \Phi uv^T \in \reals^{l\times n} . \label{sk3}
\end{align}
This observation allows us to form the sketch $Y^C$ and $Y^R$ from the stream of updates.

We then reconstruct $X$ and get the reconstructed matrix $\hat{X}$ by
\begin{align}
\label{re}
Y^C = QR, \quad  B = (\Phi Q)^{\dagger} Y^R, \quad \hat{X}= Q[B]_r,
\end{align}
where $QR$ is the $QR$ factorization of $Y^C$ and $[\cdot]_r$ returns the best rank $r$ approximation in Frobenius norm.
Specifically, the best rank $r$ approximation of a matrix $Z$ is $U\Sigma V^T$,
where $U$ and $V$ are right and left singular vectors corresponding to the $r$
largest singular values of $Z$ and $\Sigma$ is a diagonal matrix with $r$ largest singular values of $Z$. Note the matrix $R$ is not used.

The following theorem \cite[Theorem 5.1]{sketch} guarantees that the resulting reconstruction
approximates $X$ well if $X$ is approximately low rank.

\begin{theorem} \label{sk2}
	Fix a target rank $r$. Let $X$ be a matrix, and let $(Y^C,Y^R)$ be a sketch as described in equation \eqref{sk1}.
	The procedure \eqref{re} yields a rank-$r$ matrix $\hat{X}$ with
	\[\Expect \|X-\hat{X}\|_F \leq 3\sqrt{2} \|X-[X]_r\|_F.\]
\end{theorem}

In the paper \cite{sketchy}, this matrix sketching procedure is combined with the original Frank-Wolfe (Algorithm \ref{fw}).
We show here that it also works well with \SVRF, the stochastic version of Frank-Wolfe
and an approximate subproblem oracle.

We use the following matrix completion problem, which is also a particular instance of Problem \eqref{op3},  to illustrate this synthesis:

\beq
\ba{ll} \label{op1}
\mbox{minimize} & f(\mymathcal{A}W):=\frac{1}{d}\sum_{i \in I } f_{i}(\mymathcal{A}W) \\
\mbox{subject to} &  \|W\|_* \leq \alpha, \\
\ea
\eeq
where $d = |I|$ is the number of elements in $I$, $W\in \reals ^{m\times n}$,
$\mymathcal{A}: \reals^{m\times n} \rightarrow \reals^{l}$ is a linear map,
and $\alpha >0$ is a given constant. By setting $f = \sum_{i\in I} f_i$ and $\mymathcal{S}=\reals^{m\times n}$, we see it is indeed a special instance of Problem \eqref{op3}.
Since \SVRF applied to problem \eqref{op1}
updates iterates $W_k$ with a rank-one update at each inner loop iteration,
the sketch matrices $Y^C$ and $Y^R$ can be updated using equation \eqref{sk3}.
In order to compute the gradient $\nabla (f\circ \mymathcal{A})(W_k)$ at $W_k$,
we can store the dual variable $z_k =\mymathcal{A}W_k$
and compute the gradient from $z_k$ as
\[
\nabla (f\circ \mymathcal{A})(W_k)= \mymathcal{A}^*(\nabla f)(z_k).
\]
Using linearity of $\mymathcal{A}$,  the dual variable can be updated as
$$z_k := (1-\gamma_k)z_{k-1}+\gamma_k \mymathcal{A}(-\alpha u_kv_k^*).$$
We can store the dual variable efficiently if $l=\mymathcal{O}(n)$, and
we can update it efficiently if the cost of applying $\mymathcal{A}$ to
a rank one matrix is $\mymathcal{O}(l)$. 
In many settings we have $l = d$, the number of samples.
This means that storing and updating the dual variable $z_k$ could be
as costly as computing the full gradient.
However, in the oversampled setting,
where $l =\mymathcal{O}(n)$ while $d\gg \mymathcal{O}(n)$, 
combining the techniques can be beneficial.
In this setting, storing $z_k$ is not too costly,
and updating $z_k$ is also efficient
so long as applying $\mymathcal{A}$ to a matrix costs $\mymathcal{O}(l)$.

The combined algorithm, \SSVRF, is shown below as Algorithm \ref{sketchcgm}.
\begin{algorithm}
	\caption{\SSVRF \label{sketchcgm}}
	\begin{algorithmic}[1]
		\STATE {\bf Input:} Objective function $f\circ \mymathcal{A} = \frac{1}{d} \sum_{i=1}^d f_i\circ \mymathcal {A}$
		\STATE{\bf Input:} Stepsize $\gamma_k$, mini-batch size $m_k$, epoch length $N_t$ and tolerance sequence $\epsilon_k$
		\STATE{\bf Input:} Target rank $r$ and maximum number of iteration $T$
		\STATE {\bf Initialize:} Set $x_{-1} =0,Y^C=0,Y^R=0$ and draw $\Phi\in \reals^{(4r+3)\times m},$ $\Psi \in \reals^{n\times (2r+1)}$ with standard normal entries.
		\FOR{$t=1,2, \ldots, T$}
		\STATE Take a snapshot $z_0 = x_{t-1}$ and compute gradient $\nabla f(z_0)$
		\FOR{$k=1$ to $N_t$}
		\STATE Compute $\tilde{\nabla}_k$, the average of $m_k$ iid samples of     $\tilde{\nabla}f(z_{k-1},z_0)$
		\STATE Compute  $u,v$ such that
		\STATE $ -\alpha \tr ((\mymathcal{A}^*\tilde{\nabla }_k)^Tuv^T) \leq \min_{\|X\|_*\leq \alpha}\tr((\mymathcal{A}^* \tilde{\nabla}_k )^T X)+\epsilon_k$  \label{maxsing}
		\STATE	Compute $h _k= \mymathcal{A}(-\alpha uv^T)$
		\STATE Update $z_k := (1-\gamma_k)z_{k-1}+\gamma_k h_k$
		\STATE Update $Y^C_k = (1-\gamma_k)Y^C_{k-1}+\gamma_k (-\alpha u v^T)\Psi$
		\STATE Update $Y^R_k=(1-\gamma_k)Y^R_{k-1}+\gamma_k\Phi (-\alpha u v^T) $
		\ENDFOR
		\STATE Set $x_t = z_{N_t}$
		\ENDFOR
		\STATE Compute $QR$ factorization of the $Y_{N_T}^C=QR$ and compute  $B = (\Phi Q)^{\dagger} Y_{N_T}^R$
		\STATE Compute the top $r$ many left and right singular vectors $U,V$ of $B$ and the diagonal matrix $\Sigma$ with top $r$ singular values.
		\STATE {\bf Output:} $(U,\Sigma,V)$.
	\end{algorithmic}

\end{algorithm}

\section{Theoretical Guarantees for \SSVRF}

The following theorems are analogous to theorems in \cite{sketchy}.
In this work, we introduce adaptations to cope with the approximate oracle and stochastic gradient.

Let us first instantiate some definitions.
We assume for each $i$, $f_i\circ \mymathcal{A}$ is $L$-smooth with respect to the Frobenius norm.
Note that the diameter of the feasible region is bounded:
\[
\sup_{\|X\|_*,\|Y\|_* \leq \alpha } \|X-Y\|_F \leq  \sup_{\|X\|_*,\|Y\|_* \leq \alpha } \|X-Y\|_* \leq 2\alpha.
\]
Hence the parameter $D$, the diameter of the feasible set in Theorem \ref{t1}, can be replaced by $2\alpha$.
For each $t$, we denote by $\hat{X}_t$ the matrix reconstructed using $Y^C_{N_t},Y^R_{N_t}$:
\begin{align*}
Y^C_{N_t} = QR, \quad  B = (\Phi Q)^{\dagger} Y^R_{N_t}, \quad \hat{X_t}= Q[B]_r.
\end{align*}
The matrix $\hat{X}_t$ can be considered as the reconstruction of $X_t$
(the snapshot, not the inner loop iterate) in \SSVRF.
We use the same parameters as in Theorem \ref{t1} with $D$ replaced by $2\alpha$ to achieve the following theoretical guarantee:

\begin{theorem}
	Suppose we apply Algorithm \ref{svrf} or \ref{svrf2} to the optimization problem \eqref{op1}
	and that for a particular realization of the stochastic gradients,
	the iterates $X_t$ converge to a matrix $X_\infty$.
	Further suppose that in Algorithm \ref{sketchcgm}, we use the same stochastic gradients.

	Then
	\[
	\lim_{t\rightarrow \infty} \Expect_{\Psi,\Phi}\|\hat{X}_t-X_\infty\|_F \leq 3\sqrt{2} \|X_\infty - [X_\infty]_r\|_F.
	\]
\end{theorem}

\begin{proof}
	The proof exactly follows the proof of \cite[Theorem 6]{sketch}.
\end{proof}

When the solution set of optimization problem \eqref{op1} contains only matrices
with rank $\leq r$,
we can prove a stronger guarantee for Algorithm \ref{sketchcgm}:

\begin{theorem}
	Suppose that the solution set $S_*$ of the optimization problem \eqref{op1} contains
	only matrices with rank $\leq r$.
	Then Algorithm \ref{sketchcgm} attains
	\[
	\lim_{t\rightarrow \infty}\Expect \dist_F(\hat{X}_t,S_*) = 0,
	\]
	where $\dist_F(X,S_*) = \inf_{Y\in S_*} \|X-Y\|_F$.
\end{theorem}

\begin{proof}
	The triangle inequality implies that
	\begin{equation*}
	\begin{array}{l}
	\Expect\dist_F(\hat{X}_t,S_*)  \leq \Expect\|\hat{X}_t -X_t\|_F +\Expect \dist_F(X_t,S_*).
	\end{array}
	\end{equation*}
	We claim that the second term, $\Expect \dist_F(X_t,S_*)$, converges to $0$.
	If so, we may conclude that the first term converges to zero by the following inequality.
	\begin{equation*}
	\begin{array} {ll}
	\Expect \|\hat{X}_t - X_t\|_F &\leq 3\sqrt{2} \Expect \|X_t - [X_t]_r\|_F\\
	& \leq 3\sqrt{2} \Expect(\dist_F(X_t,S_*)) \rightarrow 0.
	\end{array} 
	\end{equation*}
	The first inequality is Theorem \ref{sk2}, and the second bound is due to the optimality of $[X_t]_r$.

	It remains only to prove the claim $\Expect \dist_F(X_t,S_*) \to 0$.
	Let $g= f \circ \mymathcal{A}$ and $g_*$ to be the optimal value of $g$ in program \eqref{op1}.
	Now fix a number $\epsilon>0$.
	Define
	\begin{equation*}
	\begin{array}{l}
	E = \{ X\in \reals ^{m\times n}: \|X\|_* \leq \alpha \;\text{and}\; \dist_F (X,S_*) \geq \epsilon\},
	\end{array} 
	\end{equation*}
	and $v = \inf \{ g(X), X\in E\}$.
	If $E$ is empty, then $v = +\infty$.
	Otherwise, the continuous function $g$ attains the value $v$ on the compact set $E$.
	In either case, $v>g_*$ because $E$ contains no optimal point of \eqref{op1}.
	Thus
	\[
	\Prob(X_t\in E)  \leq \Prob(g(X_t) -g_* > v-g^*)\leq \frac{\Expect(g(X_t) -g^*)}{v-g^*},
	\]
	where the first inequality is due to the optimality of $v$, and the second is just the Markov inequality.
	Notice
	\begin{equation*}
	\begin{array}{ll}
	\Expect \dist_F(X_t,S_*) & = \Expect \dist_F(X_t,S_*) \indicator_{\{X_t \in E\} } + \Expect \dist_F(X_t,S_*) \indicator_{\{X_t \notin E\} } \\&\leq 2\alpha \Prob(X_t \in E )+\epsilon
	\\ & \leq 2\alpha\frac{\Expect(g(X_t) -g^*)}{v-g^*} +\epsilon,
	\end{array}{ll}
	\end{equation*}
	where the inequality is due to the definition of $E$, and the feasible region is $\|X\|_* \leq \alpha$.
	Since $\Expect(g(X_t) )\rightarrow g_*$ by Theorem \ref{t1}, we know
	$\lim_{t\rightarrow\infty}\Expect \dist_F(X_t,S_*) \leq \epsilon$ for any $\epsilon >0$.
	Thus the claim is proved.
\end{proof}

When the solution to the optimization problem \eqref{op1} is unique and the function $f$ has a strong curvature property,
we can also bound the distance to the optimal solution in expectation.

\begin{theorem}
	Fix $\kappa>0$ and $\nu \geq 1$. Suppose the unique solution $X^\star$ of \eqref{op1} has rank less than or equal to $r$ and
	\begin{align}
	f(\mymathcal{A}X) - f(\mymathcal{A}X^\star) \geq\kappa \|X-X^\star\|_F^\nu \label{st1}
	\end{align}
	for all $\|X\|_*\leq \alpha$. Then we have the error bound
	\[\Expect \|\hat{X}_t - X^\star\|_F \leq 6\Big(\frac{4\kappa^{-1}L\alpha^2(1+\delta)}{2^{t+1}}\Big)^{\frac{1}{v}}\]
	for all $t$.
\end{theorem}

\begin{proof} Let $g = f\circ \mymathcal{A}$.
	The proof of Theorem \ref{t1} tells us that
	\[\Expect (g(X_t)- g(X^\star) )\leq \frac{ LD^2(1+\delta)}{2^{t+1}}.\]
	Since the iterate $X_t$ is  feasible, the assumption in \eqref{st1} gives us
	\begin{align}
	\Expect (g(X_t)- g(X^\star) )  & \geq \kappa \Expect \|X_t - X^\star\|_F^v  \label{b1}\\
	& \geq\kappa \Expect \| X_t - [X_t]_r\|_F^v \nonumber \\
	& \geq \kappa [\Expect( \|X_t - [X_t]_r\|_F)]^v \nonumber \\
	& \geq \frac{\kappa}{(3\sqrt{2})^v} (\Expect(\|X_t -\hat{X}_t\|_F))^v \label{b2}.
	\end{align}
	The second inequality is due to the optimality of $[X_t]_r$ and $X^\star$ has rank less then $r$. The third is because of Jensen's inequality and the last is from Theorem \ref{sk2}.  We now conclude that

	\begin{align*}
	\Expect\|\hat{X}_t - X^\star\|_F  & \leq \Expect \|\hat{X}_t - X_t\| + \Expect\|X_t - X^\star\| \\
	& \leq3\sqrt {2} \Big( \frac{\kappa^{-1} LD^2(1+\delta)}{2^{t+1}} \Big)^{1/v}+ \Big (\frac{ \kappa^{-1} LD^2(1+\delta)}{2^{t+1}} \Big)^{1/v}.
	\end{align*}
The last bound follows from inequality \eqref{b1} and \eqref{b2}. To reach the final conclusion shown in the theorem, simplify the numerical constant,  use the assumption that $v\geq 1$ and note that $D\leq 2\alpha$. 
\end{proof}

\begin{acknowledgement}
This work was supported by DARPA Award FA8750-17-2-0101.
The authors are grateful for helpful discussions with Joel Tropp, Volkan Cevher, and Alp Yurtsever.
\end{acknowledgement}
\section*{Appendix}

We prove the following simple proposition about $L$-smooth functions used in Section \ref{s2}.
\begin{proposition}
	If $f$ is a real valued differentiable convex function with domain $\reals^n$
	and satisfies $\|\nabla f(x)- \nabla f(y)\|\leq L\|x-y\|$, then for all $x,y \in \reals^n$,
	$$ f(x)\leq f(y) + \nabla f(y)^T (x-y)+\frac{L}{2}\|x-y\|^2.$$
\end{proposition}
\begin{proof} The inequality follows from the following computation:
	\beq
	\ba{ll}
	f(x) - f(y) -\nabla f(y)^T(x-y)&= \int_0^1  (\nabla f(y+t(x-y))-\nabla f(y))^T(x-y)dt \\
	&\leq  \int_0^1 \|  (\nabla f(y+t(x-y))-\nabla f(y))^T(x-y)\|dt \\
	&\leq \int_0^1 \| (\nabla f(y+t(x-y))-\nabla f(y))\|\|(x-y)\|dt \\
	&\leq  \int_0^1 Lt\|x-y\|^2dt\\
	&=  \frac{L}{2}\|x-y\|^2.
	\ea
	\eeq
\end{proof}
%


\bibliographystyle{abbrv}
\bibliography{references}

\end{document}